\documentclass{article}
\usepackage{amsmath}
\usepackage{amsthm}
\usepackage{amssymb}
\usepackage[caption=false,font=footnotesize]{subfig}
\usepackage{url}
\usepackage{mathtools}
\usepackage[utf8]{inputenc}
\usepackage[T1]{fontenc}
\usepackage{hyperref}
\usepackage{bussproofs}
\usepackage{tikz}
\usepackage{tikzit}
\usepackage{tikz-cd}
\usepackage{stmaryrd}
\usepackage{authblk}

\input{tikzit.tikzdefs}

\tikzstyle{empty}=[shape=circle, tikzit fill={rgb,255: red,191; green,191; blue,191}]
\tikzstyle{dot}=[fill=black, draw=black, shape=circle]
\tikzstyle{hook}=[right hook->, draw=black, tikzit draw=magenta]
\tikzstyle{mono2}=[draw=black, >->]
\tikzstyle{epi2}=[draw=black, ->>]

\tikzstyle{to}=[draw=black, ->]
\tikzstyle{equal-arrow}=[-, double equal sign distance]
\tikzstyle{dashed-arrow}=[dashed, ->]
\tikzstyle{epi}=[draw=black, -, epi2]
\tikzstyle{mono}=[draw=black, -, mono2]
\tikzstyle{hook}=[draw=black, -, right hook->]

\theoremstyle{plain}
\newtheorem{theorem}{Theorem}[section]
\newtheorem{lemma}[theorem]{Lemma}
\newtheorem{proposition}[theorem]{Proposition}

\newtheorem{corollary}[theorem]{Corollary}

\theoremstyle{definition}
\newtheorem{definition}[theorem]{Definition}

\newtheorem{example}[theorem]{Example}

\newtheorem{construction}[theorem]{Construction}

\theoremstyle{remark}
\newtheorem{remark}[theorem]{Remark}

\title{The Category of Operator Spaces and Complete Contractions}

\author[1]{Bert Lindenhovius}
\author[2]{Vladimir Zamdzhiev}
\affil[1]{Johannes Kepler University, Institute for Mathematical Methods in Medicine and Data Based Modeling, Altenberger Str. 69, 4040 Linz, Austria}
\affil[2]{Université Paris-Saclay, CNRS, ENS Paris-Saclay, Inria, Laboratoire Méthodes Formelles, 91190, Gif-sur-Yvette, France}
\date{}

\begin{document}

\maketitle

\begin{abstract}
  We show that the category $\OS$ of operator spaces, with complete
  contractions as morphisms, is locally countably presentable. This result,
  together with its symmetric monoidal closed structure with respect to the
  projective tensor product of operator spaces, implies the existence of cofree
  (cocommutative) coalgebras with respect to the projective tensor product and
  therefore provides a mathematical model of Intuitionistic Linear Logic
  in the sense of Lafont.
\end{abstract}

\section{Introduction}
\label{sec:intro}

Operator spaces \cite{effros-ruan,blecher-merdy,pisier} are 
structures in noncommutative geometry that generalise mathematical objects such
as von Neumann algebras, C*-algebras, and operator systems, all of which are
relevant to the study of quantum information. The increased generality
provided by operator spaces allows us to cover even more spaces that are
relevant in quantum information theory (e.g. the space of trace-class operators
on a Hilbert space) and to study them with the methods and techniques from
noncommutative geometry. Operator spaces have been widely recognised as the
noncommutative (or quantised) analogue of Banach spaces and many results from
Banach space theory have been generalised to the theory of operator spaces. In
this paper, we show that an important categorical property of Banach spaces --
local presentability -- also carries over to operator spaces and we showcase
some of the mathematical implications of this fact. This is a very strong
categorical property and we show that by using it, we can prove the existence
of cofree (cocommutative) coalgebras with respect to the projective tensor
product of operator spaces. 

Locally presentable categories \cite{gabriel-ulmer,lp-categories-book,borceux2}
are particularly well-behaved categories that have a rich mathematical
structure. Every locally presentable category $\CC$ is complete, cocomplete,
well-powered, cowell-powered, has a (strong epi, mono)-factorisation system,
has an (epi, strong mono)-factorisation system, and it also enjoys many other
mathematical properties (see \cite{lp-categories-book} for a textbook account).
Another one of these pleasant properties is that locally presentable
categories have powerful adjoint functor theorems that allow us to easily
establish adjunctions. In particular, given a functor $F \colon \CC \to \DD$
between a locally presentable category $\CC$ and a locally small category
$\DD$, then it follows that $F$ is a left adjoint iff $F$ is cocontinuous. In
the sequel, we show how these adjoint functor theorems can be used to prove
some results about operator spaces that are far from obvious.

The category of primary interest in this paper is $\OS$, the category whose
objects are the operator spaces and whose morphisms are the linear completely
contractive maps between them. We would like to explain why we choose the
category $\OS$, instead of the category $\OScb$, whose objects are the operator
spaces but whose morphisms are the linear completely bounded maps. Simply put,
$\OS$ has nicer categorical properties. More specifically, $\OScb$ does not
have small coproducts, so it cannot be locally presentable, whereas $\OS$
is indeed locally presentable (and therefore has small colimits). Another
disadvantage from a categorical viewpoint is that two finite-dimensional
operator spaces are categorically isomorphic in $\OScb$ iff they are linearly
isomorphic as vector spaces -- this shows that much of the operator space
structure is ignored by the morphisms in $\OScb$. This is not the case in $\OS$
-- indeed the categorical isomorphisms in $\OS$ coincide with the completely
isometric isomorphisms of operator spaces, so they preserve all of the
structure of an (abstract) operator space. 

The commutative (or classical) counterpart of $\OS$ is the category $\Ban$, whose objects are
the Banach spaces over the complex numbers and whose morphisms are the linear
contractions between them. It is well-known that $\Ban$ is locally countably presentable
\cite{lp-categories-book}, \cite{borceux2}\footnote{In \cite{borceux2}, the proof is given for
Banach spaces over the reals, but the same proof also works for Banach spaces
over the complex numbers.}, however, there is a notable difference between the
locally presentable structures of $\OS$ and $\Ban$ that we wish to emphasise. The
Banach space $\mathbb C$ is a strong generator (in the categorical sense) in
$\Ban$, whereas the operator space $\mathbb C$ is not a strong generator in
$\OS.$ A strong generating set in $\OS$ is given by $\{T_n \ |\ n \in \mathbb
N\}$, i.e. the set of finite-dimensional trace-class operator spaces. Another
(singular) strong generator is the operator space $T(\ell_2)$ consisting of the
trace-class operators on the Hilbert space $\ell_2.$

A category $\CC$ is locally countably presentable iff $\CC$ is cocomplete and
has a strong generating set consisting of countably-presentable objects. In
order to show that each operator space $T_n$ is countably-presentable (in the
categorical sense), we use some facts about countably-directed colimits in
$\Ban$ that are provided in \cite{borceux2}, but we also prove additional
results of our own related to these colimits that we need for our later
development. Building on this, we prove that the countably-presentable objects
in $\Ban$ are precisely the separable Banach spaces (Theorem
\ref{thm:separable-banach}). This result is not surprising, but we were not
able to find a proof of this fact anywhere in the literature, so we decided to
present one. By further developing these results to the noncommutative setting
of operator spaces, we show that the countably-presentable objects in $\OS$ are
precisely the separable operator spaces (Theorem \ref{thm:presentable-objects})
and this then allows us to conclude that the aforementioned strong generators
are countably-presentable and therefore $\OS$ is locally countably presentable
(Theorem \ref{thm:os-locally-presentable}).

In order to demonstrate the mathematical utility of locally presentable
categories, we highlight a result that we think would be more difficult to
prove for operator spaces without the use of categorical methods.
If $\CC$ is a locally presentable symmetric monoidal closed
category, then there exists an adjunction
\begin{equation}
  \label{eq:coalg-adjunction}
  \stikz{free-coalgebra-adjunction.tikz}
\end{equation}
where the category $\CoCoalg$ is the category of cocommutative coalgebras over
$\CC$, the functor $U$ is the obvious forgetful one, and where $R$ is its \emph{right}
adjoint, obtained through the use of the adjoint functor theorem. The analogous
statement holds if the category $\CoCoalg$ is replaced by the category $\Coalg$
of coalgebras (not necessarily cocommutative) over $\CC.$ In fact, the
categories $\CoCoalg$ and $\Coalg$ are also locally presentable, the category
$\Coalg$ is symmetric monoidal closed and $\CoCoalg$ is even cartesian closed.
The interested reader may find proofs of these results in
\cite{porst-coalgebras}. This makes the adjoint situation in
\eqref{eq:coalg-adjunction} very nice from a categorical viewpoint and this
furthermore gives us the structure of a mathematical model of Intuitionistic
Linear Logic \cite{linear-logic} in the sense described by Lafont
\cite{lafont-thesis}. Using the aforementioned adjunctions, it follows that if
$A$ is an object of $\CC$, then $RA$ is the cofree (cocommutative)
coalgebra\footnote{Some authors refer to $RA$ as the free (commutative)
coalgebra, but the mathematical meaning coincides with the one we present
here.} over $A$. The mathematical meaning of the cofreeness of $RA$ is given by
the universal property of the adjunction \eqref{eq:coalg-adjunction}. Let us
now consider a few choices for the category $\CC$ that might be more familiar
to readers with a background in functional analysis and linear algebra.

Taking $\CC = \Vect$, the category of vector spaces over $\mathbb C$ and linear maps, together with its usual closed monoidal structure given by
the algebraic tensor product $\otimes$, we recover the usual notion of cofree
(cocommutative) coalgebra. The construction of these cofree (cocommutative)
coalgebras is already known and it is described by Sweedler \cite{sweedler}.
Sweedler does not make use of the adjoint functor theorem or any
explicit use of category theory. Taking $\CC = \Ban$ together with its closed
monoidal structure given by the projective tensor product, we recover the
existence of the cofree (cocommutative) coalgebras with respect to this tensor
product. Other authors \cite{fox-thesis,barr-coalgebras} have already made this
observation, but their proofs also use the locally presentable structure of the
category $\Ban$. To the best of our knowledge, no other proof of this result is
known. Taking $\CC = \OS$ together with its closed monoidal structure given by
projective tensor product of operator spaces, we see that cofree
(cocommutative) coalgebras with respect to this tensor product exist. This is a
novel result that now follows from \cite{porst-coalgebras} using the locally
presentable structure of $\OS$ (proven in this paper) together with its closed
symmetric monoidal structure given by the projective tensor product of operator
spaces.

In Section \ref{sec:categorical-background}, we provide background on locally
presentable categories. In Section \ref{sec:banach}, we recall the construction
of countably-directed colimits in $\Ban$ and we establish some new results that
allow us to characterise the countably-presentable objects in $\Ban$ as exactly
the separable ones (Theorem \ref{thm:separable-banach}).  In Section
\ref{sec:os}, we prove the main results of the paper, namely that the category
$\OS$ is locally countably presentable (Theorem
\ref{thm:os-locally-presentable}) with a characterisation of the
countably-presentable objects being exactly the separable operator spaces
(Theorem \ref{thm:presentable-objects}). In Section \ref{sec:coalgebras}, we
use the local presentability of $\OS$ to prove the existence of cofree
(cocommutative) coalgebras of operator spaces with respect to the projective
tensor product. Finally, in Section \ref{sec:conclusion}, we provide some
concluding remarks.

\section{Categorical Preliminaries}
\label{sec:categorical-background}

In this section we recall some background on locally presentable categories
that we need for our development. We assume prior knowledge of basic
categorical notions such as functors, (co)limits, natural transformations, etc.
Locally presentable categories were originally introduced by Gabriel and Ulmer
\cite{gabriel-ulmer}. The background that we present here is mostly based on
the textbook accounts in \cite{lp-categories-book,borceux2}.

\begin{definition}[$\alpha$-directed Poset]
  A partially ordered set (poset) $\Lambda$ is \emph{directed} whenever every finite
  subset $X \subseteq \Lambda$ has an upper bound in $\Lambda.$ If $\alpha$ is a regular
  cardinal, we say that a poset $\Lambda$ is \emph{$\alpha$-directed} whenever every
  subset $X \subseteq \Lambda,$ with cardinality strictly smaller than $\alpha$, has
  an upper bound in $\Lambda.$
\end{definition}

In particular, the notion of directed set coincides with that of an
$\aleph_0$-directed set. We also say that an $\aleph_1$-directed set is
\emph{countably directed}. Indeed, an equivalent way to say that $\Lambda$ is
$\aleph_1$-directed is to require that every countable subset $X \subseteq \Lambda$
has an upper bound in $\Lambda.$ We also wish to note that the notion of
$\alpha$-directedness does not impose a limit on the cardinality of $\Lambda$, but
rather it is a condition regarding the shape of $\Lambda.$ For instance,
a countably-directed poset is indeed directed, but it does not have to be countable.

\begin{example}
  The natural numbers with the usual order $(\mathbb N, \leq)$ is a directed
  poset, but it is not countably-directed. The unit interval with its standard
  order $([0,1], \leq)$ is countably-directed. Another example is the poset
  $(\mathbb N^{\top}, \leq)$ consisting of the natural numbers with a freely
  added top element $\top$ with the obvious order.
\end{example}

\begin{definition}[$\alpha$-directed Diagram]
  \label{def:countably-directed-colimit}
  An \emph{$\alpha$-directed diagram} in a category $\CC$ is a functor $D
  \colon \Lambda \to \CC,$ where $\Lambda$ is an $\alpha$-directed poset viewed
  as a posetal category in the obvious way. When $\alpha = \aleph_0$, we say
  that $D$ is a \emph{directed diagram} and when $\alpha = \aleph_1$ we say that $D$
  is a \emph{countably-directed} diagram.
\end{definition}

Our proof strategy for showing the local presentability of $\OS$
relies on giving concrete descriptions of countably-directed colimits
over operator spaces, so we introduce some notation for working with such
colimits. Given an $\alpha$-directed diagram $D \colon \Lambda \to \CC$ and
elements $\lambda, \kappa \in \Lambda$ such that $\lambda \leq \kappa$, we
write $D_\lambda \eqdef D(\lambda) \in \Ob(\CC)$ for the object assignment of the diagram in $\CC$
and we write $D_{\lambda,\kappa} \eqdef D(\lambda \leq \kappa) \colon D_\lambda \to D_\kappa$
for the morphism assignment of the diagram in $\CC.$ A colimiting cocone $(C, \{c_\lambda \colon D_\lambda \to C\}_{\lambda \in \Lambda} )$
of such a diagram is called an $\alpha$-\emph{directed colimit}.

An important concept in the theory of locally presentable categories is the
idea of an $\alpha$-presentable object that we define next. These objects are
particularly well-behaved in the theory.

\begin{definition}
  \label{def:presentable-object}
  Let $\alpha$ be a regular cardinal. An object $A$ of a category $\CC$ is
  \emph{$\alpha$-presentable} whenever the hom-functor $\CC(A,-) \colon \CC \to
  \Set$ preserves $\alpha$-directed colimits. When $\alpha = \aleph_0$ we say
  that $A$ is \emph{finitely-presentable} and when $\alpha = \aleph_1$ we say
  that $A$ is \emph{countably-presentable}.
\end{definition}

\begin{example}
  \label{ex:presentable-objects}
  In the category $\Set$, the finitely-presentable objects are the
  finite sets and the countably-presentable objects are the countable
  sets. In the category $\Vect_{\mathbb K}$ of vector spaces over a field
  $\mathbb K$, the finitely-presentable objects are the finite-dimensional
  vector spaces, whereas the countably-presentable objects are the
  countably-dimensional vector spaces \cite{finitary-functors}. However,
  this correspondence is not as straightforward in other categories. In
  the category $\Top$ of topological spaces (with continuous maps as morphisms),
  the finitely-presentable objects are the finite spaces equipped with the discrete
  topology. In the category $\Ban$, the only finitely-presentable object is the
  zero-dimensional Banach space (see Section \ref{sec:banach}). Similarly, the
  only finitely-presentable object in $\OS$ is the zero-dimensional operator
  space (see Section \ref{sec:os}). In the sequel we prove that the countably-presentable
  objects in $\OS$ ($\Ban$) are precisely the separable operator spaces (Banach spaces).
\end{example}

For our development, countably-presentable objects are particularly important.
An equivalent way of saying that $A$ is countably-presentable is to require
that the functor $\CC(A,-) \colon \CC \to \Set$ preserves countably-directed
colimits. Furthermore, writing the condition in Definition
\ref{def:presentable-object} in more detail, we find that it can be formulated
within the category $\CC$ entirely. This is made precise by the following
well-known proposition (see \cite[Proposition 5.1.3]{borceux2}\footnote{This is
stated for $\alpha$-filtered colimits, but the statement also holds for
$\alpha$-directed colimits.}).

\begin{proposition}
  \label{prop:presentable-object}
  An object $A$ of a category $\CC$ is $\alpha-$presentable iff for every
  $\alpha$-directed diagram $D \colon \Lambda \to \CC$, for every colimiting
  cocone $(C, \{ c_\lambda \colon D_\lambda \to C\}_{\lambda \in \Lambda})$ of
  $D$, and for every morphism $f \colon A \to C,$ there must exist $\lambda \in
  \Lambda$, such that:
  \begin{enumerate}
    \item $f = c_\lambda \circ g$ for some $g \colon A \to D_\lambda$; and
    \item this factorisation is essentially unique in the following sense: for any two morphisms $g,g' \colon A \to D_\lambda$, if $c_\lambda \circ g = c_\lambda \circ g'$, then there exists $\tau \in \Lambda$ such that
      $\lambda \leq \tau$ and such that $D_{\lambda,\tau} \circ g = D_{\lambda,\tau} \circ g'.$
  \end{enumerate}
\end{proposition}

Locally presentable categories may be defined in many equivalent ways. A
standard way of doing so is given by the next definition (see \cite[pp. 21]{lp-categories-book}).

\begin{definition}[Locally $\alpha$-Presentable Category]
\label{def:locally-presentable}
  Let $\alpha$ be a regular cardinal. A category $\CC$ is \emph{locally $\alpha$-presentable} if:
  \begin{enumerate}
    \item $\CC$ is cocomplete;
    \item every object of $\CC$ is an $\alpha$-directed colimit of $\alpha$-presentable objects;
    \item there exists, modulo isomorphism, only a (small) set of $\alpha$-presentable objects of $\CC.$
  \end{enumerate}
  A category $\CC$ is \emph{locally presentable} if it is locally
  $\alpha$-presentable for some regular cardinal $\alpha.$
\end{definition}

We also present another equivalent way of defining locally $\alpha$-presentable
categories. The equivalent definition allows us to more easily highlight some
important differences between $\Ban$ and $\OS$ in terms of their structure as
locally presentable categories. The most notable difference is given by the
\emph{strong generators} of the categories, a notion that we recall next.

\begin{definition}
  A \emph{generating set} for a category $\CC$ is a (small) set of objects
  $\mathcal S \subseteq \Ob(\CC)$, such that for any pair of \emph{distinct}
  parallel morphisms $f, g \colon A \to B$ of $\CC$, there exists an object $S
  \in \mathcal S$ and a morphism $s \colon S \to A$ with the property that $f
  \circ s \neq g \circ s.$ A generating set $\mathcal S$ is \emph{strong}
  whenever the following condition holds: for any proper monomorphism $m \colon
  A \to B$ (i.e. a monomorphism that is not an isomorphism), there exists an
  object $S \in \mathcal S$ and a morphism $f \colon S \to B$ which does not
  factorise through $m$. An object $S$ is called a \emph{(strong) generator} if
  the singleton $\{ S \}$ is a (strong) generating family.
\end{definition}

\begin{example}
  \label{ex:strong-generators}
  In the category $\Set$, the terminal object $1$ (any singleton set) is a
  strong generator. In the category $\Vect_{\mathbb K}$, the field $\mathbb K$
  is a strong generator. In the category $\mathbf{Top}$, the terminal object 1
  (a singleton set with the unique choice of topology) is a generator, but it
  is not strong. In fact, $\mathbf{Top}$ does not have a strong generating set
  at all. In the category $\Ban$, the complex numbers $\mathbb C$ is a strong
  generator, whereas in $\OS$, the operator space $\mathbb C$ is a generator,
  but it is not strong. In the sequel, we prove that the set $\{T_n \ |\ n \in
  \mathbb N\}$, i.e. the finite-dimensional trace-class operator spaces, is a
  strong generating set for $\OS$. We also prove that the operator space
  $T(\ell_2)$, i.e. the trace-class operators on $\ell_2$, is a strong
  generator.
\end{example}

Another well-known way of introducing locally presentable categories is given
by the following proposition.

\begin{proposition}
  \label{prop:locally-presentable-by-generators}
  Let $\CC$ be a cocomplete category. Then $\CC$ is $\alpha$-presentable iff
  $\CC$ has a strong generating set consisting of $\alpha$-presentable objects.
\end{proposition}

We also consider special kinds of monomorphisms and epimorphisms that we
recall next.

\begin{definition}
  A \emph{regular monomorphism} in a category $\CC$ is a morphism $m \colon E
  \to X$ which is the equaliser of a pair of morphisms $f,g \colon X \to Y.$
\end{definition}

\begin{definition}
  A \emph{regular epimorphism} in a category $\CC$ is a morphism $e \colon Y
  \to Q$ which is the coequaliser of a pair of morphisms $f,g \colon X \to Y.$
\end{definition}

It follows easily from the definition that a regular monomorphism (epimorphism)
is indeed a monomorphism (epimorphism).

\begin{definition}
  A \emph{strong monomorphism} in a category $\CC$ is a monomorphism $m \colon C \to D$, such that for any epimorphism $e \colon A \to B$
  and morphisms $f \colon A \to C$ and $g \colon B \to D$ that make the following diagram
  \[ \stikz{strong.tikz} \]
  commute, there exists a (necessarily unique) morphism $d \colon B \to C$ such that the diagram
  \[ \stikz{strong2.tikz} \]
  commutes.
\end{definition}

The dual notion is recalled next.

\begin{definition}
  A \emph{strong epimorphism} in a category $\CC$ is an epimorphism $e \colon A \to B$, such that for any monomorphism $m \colon C \to D$
  and morphisms $f \colon A \to C$ and $g \colon B \to D$ that make the following diagram
  \[ \stikz{strong.tikz} \]
  commute, there exists a (necessarily unique) morphism $d \colon B \to C$ such that the diagram
  \[ \stikz{strong2.tikz} \]
  commutes.
\end{definition}

It is well-known that every regular monomorphism is also a strong monomorphism.
Dually, every regular epimorphism is also a strong epimorphism.

\begin{definition}
  A monomorphism $m \colon A \to B$ in category $\CC$ is called \emph{extremal}
  if it enjoys the following property: for every factorisation $m = g \circ e$,
  where $e$ is an epimorphism, it follows that $e$ is an isomorphism.
\end{definition}

\begin{definition}
  An epimorphism $e \colon A \to B$ in a category $\CC$ is called
  \emph{extremal} if it enjoys the following property: for every factorisation
  $e = m \circ f$, where $m$ is a monomorphism, it follows that $m$ is an
  isomorphism.
\end{definition}

Every strong monomorphism is also extremal and the same is true for the epimorphic counterparts. Therefore, for monomorphisms/epimorphisms in a category $\CC$, we have
\[
  \text{regular} \quad \Longrightarrow \quad \text{strong} \quad \Longrightarrow \quad \text{extremal}
\]
and if $\CC$ is locally presentable, then these classes of monomorphisms/epimorphisms coincide:
\[
  \text{regular} \quad \iff \quad \text{strong} \quad \iff \quad \text{extremal}
\]

\begin{example}
  In $\Ban$, which is locally presentable, the regular/strong/extremal
  monomorphisms coincide with the linear isometries and the
  regular/strong/extremal epimorphisms coincide with the linear quotient maps,
  i.e. linear contractions $q \colon X \to Y$ which map the open ball of $X$
  \emph{onto} the open ball of $Y$. We show in the sequel that the situation is
  completely analogous in $\OS$: regular/strong/extremal monomorphisms coincide
  with the linear complete isometries and the regular/strong/extremal
  epimorphisms coincide with the linear complete quotient maps.
\end{example}

\section{Local Presentability of Banach Spaces}
\label{sec:banach}

It is already known that $\Ban$ is locally countably presentable
\cite{lp-categories-book,borceux2}. We know of two different proofs of this
fact that we wish to discuss in relation to our proof of local presentability
of $\OS$. The proof presented in \cite[Example 1.48]{lp-categories-book} shows
that $\Ban$ is equivalent to a reflective subcategory of $\mathbf{Tot}$,
the category of totally convex spaces. The latter category is locally countably
presentable and the proof follows (using categorical arguments) from this and
the closure under countably-directed colimits of the subcategory. We do not
think that $\OS$ is equivalent to a reflective subcategory of
$\mathbf{Tot}$, so we do not think that this proof strategy would work.  In
fact, $\Ban$ is equivalent to a reflective subcategory of $\OS$ (see
Section \ref{sec:os}). Perhaps one can construct a different locally
presentable category that serves the role of $\mathbf{Tot}$ for operator
spaces, but we do not pursue this approach here. Instead, we opt for a more
direct proof that aligns with the approach in \cite{borceux2}.

The proof of local presentability of $\Ban$ in \cite[5.2.2e]{borceux2} gives an
explicit construction for countably-directed colimits and shows that the strong
generator $\mathbb R$ (Banach spaces are considered over the reals there) is
countably-presentable by showing that $\Ban(\mathbb R, -)$ preserves the
aforementioned colimits. Our proof strategy for showing local presentability of
$\OS$ is similar. The only significant difference is in the strong generators,
so we have to establish the countable-presentability of a greater range of
objects. The choice of field ($\mathbb R$ vs $\mathbb C$) is not important for
the proof in $\Ban$.

The purpose of this section is to prove some additional properties of
countably-directed colimits in $\Ban$ that are not available in the literature
(to the best of our knowledge) and that we need for our subsequent development
for the locally presentable structure of $\OS.$ Most notably, we use
Proposition \ref{prop:separable Banach spaces are presentable} in the next
section in order to characterise the countably-presentable objects in $\OS.$ In
preparation for this, we also characterise the countably-presentable objects of
$\Ban$ and we show that they coincide with the separable Banach spaces (Theorem
\ref{thm:separable-banach}). We have not been able to find a proof of this
(unsurprising) result in the literature.

\begin{remark}
  Before we begin with our technical development, we note that we adopt the
  convention that the least natural number is $1$, as it is common in
  functional analysis. In order to avoid cumbersome repetition, all maps
  between vector spaces, Banach spaces, and operator spaces are implicitly assumed
  to be linear, i.e. we may simply write ``contraction'' instead of ``linear
  contraction'', ``isometry'' instead of ``linear isometry'', etc.
\end{remark}

Let us start by explaining why $\Ban$ is not locally \emph{finitely}
presentable.  The main reason is in the way that directed colimits (that are
not necessarily countably-directed) are constructed. Suppose that $D \colon
\Lambda \to \Ban$ is a directed diagram where each morphism $D_{\lambda,\tau}
\colon D_\lambda \to D_\tau$ is given by (isometric) subspace inclusion.
Consider the space
\[
  C \eqdef \bigcup_{\lambda \in \Lambda} D_\lambda
\]
together with the norm $\norm{c} \eqdef \norm{c}_\lambda,$ where
$c \in D_\lambda$ and $\norm{\cdot}_\lambda$ is the norm of $D_\lambda.$
This definition is independent of the choice of $\lambda$ and after defining
the vector space structure of $C$ in the obvious way, we can see that $C$ is a
normed space.  Then, the colimit of $D$ is the \emph{completion} $\overline C$
of the normed space $C$, together with the obvious subspace inclusions
$D_\lambda \subset \overline C.$ The $D_\lambda$'s can be chosen such that
$C \subsetneq \overline C$ and it is now easy to prove that the only
finitely-presentable object is the zero-dimensional Banach space. Indeed, if
$X$ has dimension greater than zero, then we can construct a linear contraction
$f \colon X \to \overline C$ such that its image $f[X]$ contains some of the
newly added limit points in $\overline C$, e.g. by using the Hahn-Banach theorem.
This means that $f$ does not factorise through any of the inclusions $D_\lambda
\subset \overline C$  and therefore any non-zero Banach space $X$ cannot be
finitely-presentable (see Proposition \ref{prop:presentable-object}).

The reason that countably-directed colimits in $\Ban$ and $\OS$ behave better
categorically compared to merely directed ones, is because the \emph{shape} of
the former (not their cardinality) ensures that there is no need to take a
completion of the space as we did above. Next, we recall the construction of
countably-directed colimits in $\Ban$ from \cite[5.2.2e]{borceux2}. Note that
the morphisms $D_{\lambda,\kappa}$ are linear contractions and not necessarily
isometries as in the argument we presented above.

\begin{construction}[Countably-directed colimits in $\mathbf{Ban}$]
  \label{constr:aleph1-directed-colimit-Ban}
  Let $D \colon \Lambda \to \mathbf{Ban}$ be a countably-directed diagram in $\mathbf{Ban}$. We write $D_\lambda \eqdef D(\lambda)$ for the objects of the diagram in $\mathbf{Ban}$ and we write
  $D_{\lambda, \kappa} \eqdef D(\lambda \leq \kappa) \colon D_\lambda \to D_\kappa$ for its morphisms in $\mathbf{Ban}.$ We begin by constructing its colimit as a set: first, we
  construct the disjoint union $S \eqdef \coprod_\lambda D_\lambda$ of the underlying sets, with canonical injections $s_\lambda \colon D_\lambda \to S;$ then we take the quotient
  $C \eqdef S / \sim$ with respect to the equivalence relation
  \[ s_\lambda(x) \sim s_\kappa(y) \text{ iff } \exists \tau \in \Lambda, \text{ such that } \lambda \leq \tau \geq \kappa \text{ and } D_{\lambda, \tau}(x) = D_{\kappa, \tau}(y) . \]
  Then, the colimiting cocone of $D$ as a diagram in $\Set$ is given by $(C, \{ c_\lambda \colon D_\lambda \to C \}_{\lambda \in \Lambda} ),$ where $c_\lambda(x) \eqdef [s_\lambda(x)]$ maps elements into the appropriate equivalence class.
  Taking the vector space structure into account, the zero vector of $C$ is given by $0 \eqdef c_\lambda(0)$, where $\lambda \in \Lambda$ may be chosen arbitrarily, scalar multiplication
  is defined by $a[s_\lambda(v)] \eqdef [s_\lambda(av)]$, and addition is defined by
  \[ [s_\lambda(v)] + [s_\kappa(w)] \eqdef [s_\tau(D_{\lambda,\tau}(v) + D_{\kappa,\tau}(w))], \text{ where } \lambda \leq \tau \geq \kappa, \]
  and these assignments are indeed well-defined. Equipped with this structure,
  the set $C$ becomes a vector space and it then follows that $c_\lambda \colon
  D_\lambda \to C$ are linear maps and $(C, \{ c_\lambda \}_{\lambda \in \Lambda} )$ is the colimit of $D$
  in the category $\Vect.$ So far, we have only used the fact that $D$ is
  directed and the construction so far gives the aforementioned colimits in
  this case as well. The stronger countable-directedness property of $D$ only
  becomes relevant for the Banach space (and norm) structure of $C$, which we
  describe next.  Taking the Banach space structure into account, we define a
  norm on $C$ by
  \begin{align*}
    \norm{ v } 
    &\eqdef \inf\ \{ \norm{w} \ |\ \kappa \in \Lambda, w \in D_\kappa,  c_\kappa(w) = v \} \\
    &= \inf\ \{ \norm{w} \ |\ \kappa \in \Lambda, w \in D_\kappa,  [s_\kappa(w)] = v \} ,
  \end{align*}
  i.e. by taking the infimum of all norms of vectors in the same equivalence
  class. It is obvious from the above assignment that each $c_\kappa$ is a
  contraction\footnote{There is a slight abuse in terminology here, because one
  would first have to prove that $\norm{\cdot}$ is a norm on $C$, but the
  meaning should be clear regardless.}.
\end{construction}

\begin{lemma}\label{lem:increasing}
Let $C$ be obtained via Construction \ref{constr:aleph1-directed-colimit-Ban}, and let $v\in C$. Then:
\begin{itemize}
    \item[(a)] for each $\epsilon>0$, there is some $\kappa\in\Lambda$ and some $w\in D_\kappa$ such that $c_\kappa(w)=v$ and $\|v\|\leq\|w\|<\|v\|+\epsilon$;
    \item[(b)] for each $\lambda\geq\kappa$ in $\Lambda$, the element $w'=D_{\kappa,\lambda}(w)$ in $D_\lambda$ satisfies  $c_\lambda(w')=v$ and $\|v\|\leq \|w'\|\leq \|w\|$.
\end{itemize}
\end{lemma}
\begin{proof}
    By definition of $\|v\|$, there is some $\kappa\in\Lambda$ and $w\in D_\kappa$ such that $c_\kappa(w)=v$ and $\|w\|-\|v\|<\epsilon$. The
    definition of $\|v\|$ also implies that $\|v\|\leq \|w\|$ and therefore $\|v\|\leq\|w\|<\|v\|+\epsilon$. 
    
    For (b), let $\lambda\geq\kappa$, and let $w'=D_{\kappa,\lambda}(w)$. Since $D_{\kappa,\lambda}$ is a contraction, we have $\|w'\|\leq \|w\|$.
    Moreover, we have $c_\lambda(w') = (c_\lambda\circ D_{\kappa,\lambda})(w)=c_{\kappa}(w)=v$, from which $\|v\|\leq\|w'\|$ follows, again by definition of $\|v\|$.
\end{proof}

The following property of countably-directed colimits in $\Ban$ has been
claimed in \cite[pp. 70]{lp-categories-book} without proof. We provide a proof
for completeness.

\begin{lemma}\label{lem:norm v is norm w}
  Let $C$ be obtained via Construction \ref{constr:aleph1-directed-colimit-Ban}.
  For each $v\in C$ there is some $\lambda\in\Lambda$ and some $w\in D_\lambda$
  such that $c_\lambda(w)=v$ and $\|w\|=\|v\|$.    
\end{lemma}
\begin{proof}
  By Lemma \ref{lem:increasing}(a), for each $n\in\mathbb N$ there is some $\kappa_n\in\Lambda$ and some $w_n\in D_{\kappa_n}$ such that $c_{\kappa_n}(w_n)=v$ and $\|v\|\leq\|w_n\|<\|v\|+\frac{1}{n}$.
  Let $n,m\in\mathbb N$. Then $[s_{\kappa_n}(w_n)]=c_{\kappa_n}(w_n)=v=c_{\kappa_m}(w_m)=[s_{\kappa_m}(w_m)]$, so $s_{\kappa_n}(w_n)\sim s_{\kappa_m}(w_m)$,
  which implies that there is some $\kappa_{(n,m)} \in \Lambda$ such that $ \kappa_n \leq \kappa_{(n,m)} \geq \kappa_m$ and such that $D_{\kappa_n,\kappa_{(n,m)}}(w_n)=D_{\kappa_m,\kappa_{(n,m)}}(w_m)$.
  By countable-directedness of the diagram, there is now some $\lambda\geq\kappa_{(n,m)}$ for each $n,m\in\mathbb N$.
  Then for each $n,m\in\mathbb N$, we have $D_{\kappa_n,\lambda}(w_n)=D_{\kappa_m,\lambda}(w_m)$, because
  \begin{align*}
    D_{\kappa_n,\lambda}(w_n) &=
    \left( D_{\kappa_{(n,m)},\lambda} \circ D_{\kappa_n,\kappa_{(n,m)}} \right)(w_n) \\
    &= \left(D_{\kappa_{(n,m)},\lambda} \circ D_{\kappa_m,\kappa_{(n,m)}} \right)(w_m) \\
    &= D_{\kappa_m,\lambda}(w_m).
  \end{align*}
  So, let $w\in D_\lambda$ be the element such that $w=D_{\kappa_n,\lambda}(w_n)$ for each $n\in\mathbb N$.
  Then it follows from Lemma \ref{lem:increasing}(b) that $c_\lambda(w)=v$ and $\|v\|\leq\|w\|\leq\|w_n\|<\|v\|+\frac{1}{n}$ for each $n\in\mathbb N$. Hence, we must have $\|v\|=\|w\|$. 
\end{proof}

Next, we present an important property of countably-directed colimits in
$\Ban$. We have not seen this claimed in the literature, but this proposition
is essential for proving that separable Banach spaces are
countably-presentable.

\begin{proposition}\label{prop:distances in C and D_lambda}
  Let $\{v_n \ |\ n\in\mathbb N\}$ be a countable subset of $C$ from Construction \ref{constr:aleph1-directed-colimit-Ban}. Then there is some $\lambda\in\Lambda$ and a countable subset
  $\{w_n \ |\ n\in\mathbb N\} \subseteq D_\lambda$ such that $c_\lambda(w_n)=v_n$ and $\|v_n-v_m\|=\|w_n-w_m\|$ for each $n,m\in\mathbb N$.  
\end{proposition}
\begin{proof}
  Let $v_0\eqdef 0$. By Lemma \ref{lem:norm v is norm w}, for each $n,m\in\mathbb Z_{\geq 0}$, we can find some $\kappa_{(n,m)}\in\Lambda$ and some $x_{n,m}\in D_{\kappa_{(n,m)}}$ such that
  $c_{\kappa_{(n,m)}}(x_{n,m})=v_n-v_m$ and $\|x_{n,m}\|=\|v_n-v_m\|$.
  By countable-directedness, we can find $\kappa\geq\kappa_{(n,m)}$ for each $n,m\in\mathbb Z_{\geq 0}$ and
  we define $y_{n,m} \eqdef D_{\kappa_{(n,m)},\kappa}(x_{n,m}) \in D_\kappa$.
  Since $D_{\kappa_{(n,m)},\kappa}$ is a contraction, we have $\|y_{n,m}\|\leq\|x_{n,m}\|=\|v_n-v_m\|$ for each $n,m\in\mathbb N$.
  Note that
  \[ c_{\kappa}(y_{n,m}) = \left( c_{\kappa} \circ D_{\kappa_{(n,m)},\kappa} \right) (x_{n,m}) = c_{\kappa_{(n,m)}}(x_{n,m}) = v_n-v_m . \]
  Hence, we have
  \[ c_{\kappa}(y_{n,0}-y_{m,0})=c_\kappa(y_{n,0})-c_{\kappa}(y_{m,0})=v_n-v_m=c_{\kappa}(y_{n,m}) \] 
  and therefore $s_{\kappa}( y_{n,0}-y_{m,0} ) \sim s_\kappa(y_{n,m})$ in $\coprod_{\lambda\in\Lambda}D_\lambda$.
  This means that for each $n,m\in\mathbb Z_{\geq 0}$ there is some $\lambda_{(n,m)}\geq \kappa$ such that $D_{\kappa,\lambda_{(n,m)}}(y_{n,m})=D_{\kappa,\lambda_{(n,m)}}(y_{n,0}-y_{m,0})$.
  Again, by countable-directedness, there is some $\lambda\geq\lambda_{(n,m)}$ for each $n,m\in\mathbb Z_{\geq 0}$.
  Now, for each $n\in\mathbb N$, let $w_n \eqdef D_{\kappa,\lambda}(y_{n,0}) \in D_\lambda$.
  Then $c_{\lambda}(w_n) = \left( c_{\lambda}\circ D_{\kappa,\lambda} \right) (y_{n,0})=c_{\kappa}(y_{n,0})=v_n-v_0=v_n$, whereas for each $n,m\in\mathbb Z_{\geq 0}$, we have
  \begin{align*}
    \|w_n-w_m\| &= \|D_{\kappa,\lambda}(y_{n,0}-y_{m,0})\| \\
                &= \| \left( D_{\lambda_{(n,m)},\lambda}\circ D_{\kappa,\lambda_{(n,m)}} \right) (y_{n,0}-y_{m,0})\| \\
                & \leq\|D_{\kappa,\lambda_{(n,m)}}(y_{n,0}-y_{m,0})\| \\
                &=\|D_{\kappa,\lambda_{(n,m)}}(y_{n,m})\|\leq\|y_{n,m}\|\leq\|v_n-v_m\|,
  \end{align*}
  where we used that the $D_{\cdot,\cdot}$ maps are contractions. Since $c_\lambda(w_n-w_m)=c_\lambda(w_n)-c_\lambda(w_m)=v_n-v_m$, it follows by definition of $\norm{-}$ on $C$ that also
  $\|v_n-v_m\|\leq\|w_n-w_m\|$, hence we obtain $\|v_n-v_m\|=\|w_n-w_m\|$ for each $n,m\in\mathbb N$.
\end{proof}

To illustrate the usefulness of the above proposition, we show how we can
easily prove the next corollary (that is already known) by using it.

\begin{corollary}
  \label{cor:colimits-ban}
  Let $D \colon \Lambda \to \Ban$ be a countably-directed diagram in $\Ban$. Then the colimit of $D$ is given by Construction \ref{constr:aleph1-directed-colimit-Ban}.
\end{corollary}
\begin{proof}
  One must show that $C$ is indeed a Banach space.
  We begin by proving that the assignment $\|\cdot\| \colon C \to [0,\infty)$
  from Construction \ref{constr:aleph1-directed-colimit-Ban} is indeed a
  norm. The only nontrivial step is showing that $\|v\|=0$ implies $v=0$. So
  assume that $v\in C$ with $\|v\|=0$. By Lemma \ref{lem:norm v is norm w},
  there is some $\lambda\in\Lambda$ and some $w\in D_\lambda$ such that
  $\|w\|=\|v\|$ and $c_\lambda(w)=v$. Since $D_\lambda$
  is a normed space and $\norm w = 0$, it follows that $w=0$ and therefore $v=c_\lambda(w)=0$. 
  
  Next, we show completeness of $C$.   Let $ (v_n)$ be a Cauchy sequence in $C$. By Proposition
  \ref{prop:distances in C and D_lambda}, we can find some $\lambda\in\Lambda$
  and a sequence $(w_n)$ in $D_\lambda$ such that  $\|w_n-w_m\|=\|v_n-v_m\|$ and $c_\lambda(w_n) = v_n$ for
  each $n,m\in\mathbb N$. Hence, $(w_n)$ must also be a Cauchy sequence in $D_\lambda$, which is complete, hence the limit
  $w \eqdef \lim_{n\to\infty}w_n$ in $D_\lambda$ exists. Let $v=c_\lambda(w)$. Then
  for each $\epsilon>0$, there is some $N\in\mathbb N$ such that for each $n \geq N$ we have
  $\|w-w_n\|<\epsilon$. Hence,
  \[ \|v-v_n\|=\|c_\lambda(w)-c_\lambda(w_n)\|=\|c_\lambda(w-w_n)\|\leq\|w-w_n\|<\epsilon, \]
  where we used that $c_\lambda$ is a contraction. This shows that
  $v=\lim_{n\to\infty}v_n$ and therefore $C$ is indeed complete.

  For the couniversal property of the colimit, we already know that the cocone $(C, \{c_\lambda\}_{\lambda \in \Lambda})$ is
  colimiting in $\Vect$. The linear maps $c_\lambda \colon D_\lambda
  \to C$ are contractive by construction, therefore $(C, \{c_\lambda\}_{\lambda \in \Lambda})$ is a cocone of $D$ in $\Ban.$ If $(Z, \{ z_\lambda \}_{\lambda \in \Lambda})$
  is another cocone for $D$, then the unique mediating map $u \colon C \to Z$ is defined in the same way
  as in $\Vect$, namely by $u([s_\lambda(x)]) \eqdef z_\lambda(x)$, and it is easy to see this is independent of the choice of $\lambda \in \Lambda$ and $x \in D_\lambda.$
  Since this is the unique mediating map (in $\Vect$), we just have to show that $u$ is a contraction.
  This follows easily using Lemma \ref{lem:norm v is norm w}, because we can choose $\kappa \in \Lambda$ and $y \in D_\kappa$, such that
  $ [s_\lambda(x)] = c_\lambda(x) = c_\kappa(y) = [s_\kappa(y)]$ and $\norm y = \norm{c_\lambda(x)}$,
  therefore
  \[ \norm{u(c_\lambda(x))} = \norm{u(c_\kappa(y))} = \norm{z_\kappa(y)} \leq \norm{y} = \norm{c_\lambda(x)} . \qedhere \]
\end{proof}

In order to show that any separable Banach space is countably-presentable, we need
to establish further properties of the colimit construction in preparation for this.
Before we formulate our next proposition, we define
$\mathbb Q[i]$ to be the Gaussian rationals $\mathbb Q+\mathbb Qi$. Given a
subset $A$ of a vector space $V$ over $\mathbb C$, we define the subset
$\mathrm{span}_{\mathbb Q[i]}(A)$ of $V$ as follows:
\[ 
  \mathrm{span}_{\mathbb Q[i]}(A)\eqdef \left\{\sum_{k=1}^n\alpha_k x_k:n\in\mathbb N,\alpha_1,\ldots,\alpha_n\in\mathbb Q[i],x_1,\ldots,x_n\in A\right\}.
\]

\begin{proposition}
  \label{prop:span-colimit}
  Let $A$ be a countable subset of $C$ from Construction \ref{constr:aleph1-directed-colimit-Ban} and let
  $B = \mathrm{span}_{\mathbb Q[i]}(A)$ be its closure under finite $\mathbb Q[i]$-linear combinations.
  Then there is some $\lambda\in\Lambda$ and a countable subset
  $V = \{v_b \ |\ b \in B\} \subseteq D_\lambda$ such that:
  \begin{itemize}
    \item $c_\lambda(v_b)=b,$ for each $b \in B$;
    \item $\|v_a-v_b\| = \| a-b \|,$ for each $a,b\in B$;
    \item $v_a+\mu v_b=v_{a+\mu b}$ for each $a,b\in B$ and each $\mu\in \mathbb Q[i]$, i.e. the subset $V$ is closed under finite $\mathbb Q[i]$-linear combinations.
  \end{itemize}
\end{proposition}
\begin{proof}
  The set $B$ is countable, because the Gaussian rations $\mathbb Q[i]$ are countable, so by Proposition
  \ref{prop:distances in C and D_lambda}, there is some $\kappa\in\Lambda$ such
  that for each $b\in B$, there is some $w_b\in D_\kappa$ with
  $b=c_\kappa(w_b)$, and such that $\|w_a-w_b\| = \| a-b \|$ for each
  $a,b\in B$. Let
  \[ R=\{(a,b,d,\mu)\in B^3\times\mathbb Q[i]: a+\mu b=d\} . \]
  Clearly, $R$ is also countable. Then for each $a,b,d\in B$ and
  $\mu\in\mathbb Q[i]$ such that $a+\mu b=d$, we have $(a,b,d,\mu)=r$ for some
  $r\in R$ and we have: 
  \[c_\kappa( w_a+\mu w_b)=c_\kappa(w_a)+\mu c_\kappa(w_b)= a + \mu b = d = c_\kappa(w_d),\]
  hence $w_a+\mu w_b$ and $w_d$ are mapped to the same equivalence class in $C.$
  So there is some
  $\kappa_r\geq\kappa$ such that $D_{\kappa,\kappa_r}(w_a+\mu
  w_b)=D_{\kappa,\kappa_r}(w_d)$. By countability of $R$ and by
  countable-directedness of $D$, there is some $\lambda\geq\kappa_r$ for each
  $r\in R$. For each $b\in B$, let $v_b\eqdef D_{\kappa,\lambda}(w_b)$. Then we
  have for each $b \in B$ that
  \[ c_\lambda(v_b)= (c_\lambda\circ D_{\kappa,\lambda})(w_b)=c_{\kappa}(w_b) = b . \]
  Furthermore, for for each $a,b\in B$, we also have
  \[ \|v_a-v_b\|=\|D_{\kappa,\lambda}(w_a-w_b)\| \leq \|w_a-w_b\| = \|a-b\| = \norm{c_\lambda(v_a - v_b)} \leq \norm{v_a - v_b}, \]
  therefore $\norm{v_a - v_b} = \norm{a-b}.$
  Finally, for each $a, b \in B$ and each $\mu \in \mathbb Q[i]$, let $r = (a,b,d,\mu)\in R$, where $d = a +\mu b$. Then we have
  \begin{align*}
    v_a + \mu v_b &= D_{\kappa, \lambda}(w_a + \mu w_b) \\
    &= (D_{\kappa_r, \lambda} \circ D_{\kappa, \kappa_r})(w_a + \mu w_b) \\
    &= (D_{\kappa_r, \lambda} \circ D_{\kappa, \kappa_r})(w_d) \\
    &= D_{\kappa, \lambda}(w_d) \\
    &= v_d,
  \end{align*}
  as required.
\end{proof}

\begin{proposition}\label{prop:closed-separable}
  Let $X \subseteq C$ be a closed separable subspace of the space $C$ from Construction
  \ref{constr:aleph1-directed-colimit-Ban}. Then there is some
  $\lambda\in\Lambda$ and a closed separable subspace $X_\lambda \subseteq
  D_\lambda$ such that $c_\lambda \colon D_\lambda \to C$ (co)restricts to an
  isometric isomorphism $X_\lambda \cong X,$ i.e. such that the following diagram
  \[ \stikz{colimit-corestriction.tikz} \]
  commutes.
\end{proposition}
\begin{proof}
  Let $A \subseteq X$ be a countable dense subset of $X$ and consider $B =
  \mathrm{span}_{\mathbb Q[i]}(A) $ which is also a countable dense subset of
  $X$. By Proposition \ref{prop:span-colimit}, we can find $\lambda \in
  \Lambda$ and a countable subset $V \subseteq D_\lambda$ that enjoys the
  properties listed there. Since $\mathrm{span}_{\mathbb Q[i]}(V) = V$, then its closure $X_\lambda \eqdef \overline{V}$
  is a closed separable subspace of $D_\lambda.$ Let $h \colon X_\lambda \to X$ be the (co)restriction of $c_\lambda \colon D_\lambda \to C$ to the indicated (co)domains.
  This is justified, because for every $v \in V$, we have $h(v) = c_\lambda(v) \in B$ by the first property in Proposition \ref{prop:span-colimit} and
  then by taking closures. We now show how to construct a linear contraction $g \colon X \to X_\lambda$ which is the inverse of $h.$

  First, if $x\in X$, then
  there must be some sequence $(a_n)_{n \in \mathbb N}$ in $B$ with limit $x$, because
  $B\subseteq X$ is dense. Since $(a_n)_{n\in\mathbb N}$ converges in $X$, it
  must be a Cauchy sequence. Using the second property of Proposition \ref{prop:span-colimit}, for each $n,m\in\mathbb N$, we have
  $\|v_{a_n}-v_{a_m}\| = \|a_n-a_m\|$, and therefore
  $(v_{a_n})_{n \in \mathbb N}$ is a Cauchy sequence in $D_\lambda$.
  By completeness
  of $D_\lambda$, this sequence must have a limit, so we
  define $g(x)\eqdef\lim_{n\to\infty}v_{a_n}$.
  We prove that the value of $g(x)$ is independent of the choice of sequence in
  $B$ with limit $x$. Let $(b_n)_{n\in\mathbb N}$ be another sequence in $B$
  with limit $x$. For each $n\in\mathbb N$, we have
  \begin{align*}
    \|g(x)-v_{b_n}\| &\leq \|g(x)-v_{a_n}\|+\|v_{a_n}-v_{b_n}\| \\
                     &= \|g(x)-v_{a_n}\|+\|a_n-b_n\| \\
                     &\leq \|g(x)-v_{a_n}\|+\|a_n-x\|+\|x-b_n\|
  \end{align*}
which approaches $0$ if $n\to\infty$, hence $g(x)=\lim_{n\to\infty}v_{a_n} = \lim_{n\to\infty}v_{b_n}$.

  In order to show that $g$ is linear, let $x,y\in
  X$ and $\mu\in\mathbb C$. Then there are sequences $(a_n),(b_n)$ in $B$ and
  $(\mu_n)$ in $\mathbb Q[i]$ converging to $x$, $y$ and $\mu$, respectively,
  because $B$ is dense in $X$ and the Gaussian rationals $\mathbb Q[i]$ are
  dense in $\mathbb C$. Then $d_n\eqdef a_n+\mu_n b_n$ defines a sequence in
  $B$ converging to $x+\mu y$ in $X$. By the third property of Proposition \ref{prop:span-colimit}, we have
  $v_{a_n}+\mu_n v_{b_n}=v_{d_n}$ and therefore
  \begin{align*}
    g(x+\mu y)
    &=\lim_{n\to\infty} v_{d_n} \\
    &= \lim_{n\to\infty}( v_{a_n}+\mu_n v_{b_n}) \\
    &=\left( \lim_{n\to\infty}v_{a_n} \right)+ \left(\lim_{n\to\infty}\mu_n \right) \left(\lim_{n\to\infty}v_{b_n} \right) \\
    &=g(x)+\mu g(y),
  \end{align*}
  where we used that vector space operations in a normed vector space are
  continuous with respect to the norm topology. Therefore $g$ is linear. 

  Next, we show that $g$ is an isometry. Let $x=\lim_{n\to\infty}a_n$ in $X$ for some sequence $(a_n)_{n\in\mathbb N}$ in $B$.
  Since the norm on any normed vector space is continuous, we have $\|x\|=\lim_{n\to\infty}\|a_n\|$.
  Therefore
  \[\|g(x)\|=\left\|\lim_{n\to\infty}v_{a_n}\right\|=\lim_{n\to\infty} \|v_{a_n}\| = \lim_{n\to\infty}\|a_n\|=\|x\| , \]
  where we also used the fact that $\norm{v_{a_n}} = \norm{a_n}$ which follows by the second property of Proposition
  \ref{prop:span-colimit} (simply take one of the vectors there to be $0$).
  This means that $g$ is an isometry.

  To complete the proof, observe that we have for every $b \in B$ that $h(g(b)) = h(v_b) = b$
  and for every $v_b \in V$, we have that $g(h(v_b)) = g(b) = v_b$. It follows that
  $h \circ g = \id_{X}$ and $g \circ h = \id_{X_\lambda},$ because the
  functions on both sides of the equalities are continuous and they coincide on
  dense subsets of spaces with Hausdorff topologies. Therefore both functions
  are isometric isomorphisms and $g = h^{-1}$. The commutativity of the diagram
  is obvious.
\end{proof}

We can now show that any separable Banach space is countably-presentable in
$\mathbf{Ban}$ by using Proposition \ref{prop:distances in C and D_lambda}
again. In fact, what we prove is a slightly stronger statement, because the
maps $g$ and $g'$ in the second part of the proposition are assumed to be
bounded and not necessarily contractive. We make use of this stronger property
in the sequel when working with operator spaces (see Theorem
\ref{thm:presentable-objects}).

\begin{proposition}\label{prop:separable Banach spaces are presentable}
  Let $X$ be a separable Banach space, and let $D:\Lambda\to\mathbf{Ban}$ be a countably-directed diagram with colimit $(C, \{ c_\lambda\}_{\lambda \in \Lambda})$ as given by
  Construction \ref{constr:aleph1-directed-colimit-Ban}. Let $f\colon X\to C$ be a linear contraction. Then
  \begin{itemize}
      \item[(1)] $f=c_\lambda\circ g$  for some $\lambda\in\Lambda$ and some linear contraction $g:X\to D_\lambda$;
      \item[(2)] for each $\lambda\in\Lambda$, if $g,g':X\to D_\lambda$ are bounded linear maps such that $c_\lambda\circ g=c_\lambda\circ g'$,
        then there exists some $\tau\geq\lambda$ such that $D_{\lambda,\tau}\circ g=D_{\lambda,\tau}\circ g'$.
  \end{itemize}
\end{proposition}
\begin{proof}
  For (1), since $X$ is separable, then $Y \eqdef \overline{f[X]}$ is a closed separable subspace of $C.$
  Using Proposition \ref{prop:closed-separable}, we can find $\lambda \in \Lambda$ and a closed separable
  subspace $Y_\lambda \subseteq D_\lambda$, together with an isometric isomorphism $i \colon Y_\lambda \xrightarrow{\cong} Y$ that
  is given by (co)restricting $c_\lambda \colon D_\lambda \to C.$ We define $g \colon X \to D_\lambda$ by
  $g(x) \eqdef i^{-1}(f(x))$ which is clearly a linear contraction. It is now obvious that
  $c_\lambda \circ g = f$.

  For the second condition, let $A$ be a countable dense subset of $X.$
  Then for each $a\in A$, we have $(c_\lambda\circ g)(a)= (c_\lambda\circ g')(a)$,
  i.e. $g(a)$ and $g'(a)$ are mapped to the same equivalence class in $C$. This
  implies there exists some $\tau_a\geq\lambda$ such that
  $D_{\lambda,\tau_a}( g(a) ) = D_{\lambda,\tau_a} (g'(a))$.
  Since $A$ is countable, countable-directedness of $\Lambda$ implies the
  existence of some $\tau\geq \tau_a$ for each $a\in A$. It now follows that
  $(D_{\lambda,\tau}\circ g)(a)=(D_{\lambda,\tau}\circ g')(a)$ for each $a\in A$.
  We conclude that $D_{\lambda,\tau}\circ g=D_{\lambda,\tau}\circ g'$ because these
  functions are continuous with Hausdorff codomain and they coincide on a dense 
  subset.
\end{proof}

In order to show that any countably-presentable object in $\Ban$ is separable, the following proposition is important.

\begin{proposition}
  \label{prop:ban-self-colimit}
  Every Banach space is a countably-directed colimit (in $\Ban$) of its closed separable subspaces.
\end{proposition}
\begin{proof}
  Let $X$ be a Banach space.
  The embedding of any closed subspace of $X$ into $X$ is an isometry. Hence, if $\Lambda$ denotes the poset of all closed separable subspaces of $X$ ordered by inclusion,
  then every element of $\Lambda$ is a Banach space. Moreover, $\Lambda$ is countably-directed. To see this, let $\{ C_n \ |\ n \in \mathbb N \} \subseteq \Lambda$
  be a countable family of closed separable subspaces  and for every $n \in \mathbb N$, let $S_n \subseteq C_n$ be a countable dense subset of $C_n.$
  Standard topological results show that
  \[ C_n \subseteq \bigcup_{k \in \mathbb N} C_k = \bigcup_{k \in \mathbb N} \overline{S_k} \subseteq \overline{\bigcup_{k \in \mathbb N} S_k }
    \subseteq \overline{\mathrm{span}_{\mathbb Q[i]}\left( \bigcup_{k \in \mathbb N} S_k \right) }, \]
  where the last set is a closed separable subspace of $X$. Therefore $\Lambda$ is countably-directed. The diagram $D \colon \Lambda \to \Ban$ defined in the obvious way
  is such that $\mathrm{colim}\ D=X$ with colimiting cocone given by the subspace inclusions into $X$.
\end{proof}

We can now fully characterise the countably-presentable objects of $\Ban.$

\begin{theorem}\label{thm:separable-banach}
    A Banach space $X$ is countably-presentable in $\Ban$ iff $X$ is separable.
\end{theorem}
\begin{proof}
  ($\Leftarrow$) Combine Proposition \ref{prop:separable Banach spaces are
  presentable} and Proposition \ref{prop:presentable-object}.
    
  ($\Rightarrow$)
  Using Proposition \ref{prop:ban-self-colimit}, consider the same diagram $D$ consisting of the closed separable subspaces of $X$.
  We write $D_\lambda$ for the objects of the diagram and $c_\lambda \colon D_\lambda \subseteq X$ for the subspace inclusions of the colimiting cocone.
  This allows us to see the identity $\id \colon X \to X$ also as a map
  $\id \colon X\to\mathrm{colim}\ D$.
  Since $X$ is countably-presentable, there is some $\lambda \in \Lambda$ such
  that $c_\lambda \circ g= \id_X$ for some $g\colon X \to D_\lambda$, again by
  Proposition \ref{prop:presentable-object}.  Therefore $X = D_\lambda$ and so
  $X$ must be separable.
\end{proof}

\section{Local Presentability of Operator Spaces}
\label{sec:os}

In this section we prove the main the result of our paper, namely that the
category $\OS$ is locally countably presentable. We begin by recalling some
preliminaries on operator spaces in Subsection \ref{sub:os-preliminaries}. We
show that $\OS$ is complete and cocomplete in Subsection \ref{sub:os-colimits}. In
Subsection \ref{sub:generators} we identify the relevant strong generators for
the category $\OS$. In Subsection \ref{sub:presentable} we prove that the
countably-presentable objects in $\OS$ are precisely the separable operator
spaces. Finally, we combine these results in Subsection
\ref{sub:os-locally-presentable} to prove the local presentability of $\OS.$

\subsection{Preliminaries on Operator Spaces}
\label{sub:os-preliminaries}

We begin by recalling some background on operator spaces and we also use this
as an opportunity to fix notation. The material we present in this subsection
is standard and it is based on the textbook accounts
\cite{effros-ruan,blecher-merdy,pisier}. Note that reference
\cite{effros-ruan}, which we use extensively, is the newer (2022) version of
the book and there are some differences compared to the older edition.

\begin{definition}[Matrix Space]
  Let $V$ be a vector space. We write $\MM_n(V)$ for the vector space
  consisting of the $n \times n$ matrices with matrix entries in the vector
  space $V$ (with vector space operations defined componentwise).  When $V =
  \mathbb C$, we often write $\mathbb M_n$ for $\mathbb M_n(\mathbb C).$
\end{definition}

The vector space $\MM_n$ can be equipped with a Banach space norm in a
canonical way that we now describe.  There exists a linear isomorphism
$\MM_n \cong B(\mathbb C^n)$ with the space of (bounded) linear
operators on the Hilbert space $\mathbb C^n$. The space $B(\mathbb C^n)$ has a
canonical norm (i.e.  the operator norm) which can then be used to define a
norm on $\MM_n$ via the above isomorphism. Writing $M_n$ for the
corresponding Banach space, the aforementioned linear isomorphism now becomes
an isometric isomorphism $M_n \cong B(\mathbb C^n).$

\begin{definition}
  \label{def:operator-space}
  An \emph{abstract operator space} is a vector space $X$ together with a family of norms $\{ \norm{-}_n \colon \MM_n(X) \to [0, \infty) \ |\ n \in \mathbb N \}$, such that:
  \begin{enumerate}
    \item[(B)] The norm $\norm{-}_1 \colon \MM_1(X) \to [0,\infty)$ gives $\MM_1(X)$ the structure of a Banach space;
    \item[(M1)] $\|x\oplus y\|_{m+n}=\max\{\|x\|_m,\|y\|_n\}$
    \item[(M2)] $\|\alpha x\beta\|_m\leq\|\alpha\|\|x\|_m\|\beta\|$
  \end{enumerate}
  for each $n,m\in\mathbb N$, $x\in\MM_m(X)$, $y\in\MM_n(X)$, $\alpha,\beta\in M_{m}$. Here $x\oplus y\in \MM_{m+n}(X)$ is defined as the matrix 
  \[ 
    x\oplus y\eqdef
    \begin{bmatrix}
      x & 0\\
      0 & y
    \end{bmatrix}
  \]
  and $\alpha x \beta$ is defined through the obvious generalisation of matrix
  multiplication.
  For an operator space $X$, we write $M_n(X)$ for the normed space $(\MM_n(X),
  \norm{-}_n)$. We also write $\norm{-} \colon X \to [0, \infty)$ for the norm
  defined on $X$ through the linear isomorphism $M_1(X) \cong X.$ We
  call a family of norms that satisfies the above criteria an \emph{operator
  space structure} (OSS) on the vector space $X$.
\end{definition}

Given an operator space $X$, it is clear that $X$ is a Banach space with
respect to the norm $\norm{-}$ described above and textbook results show that
each space $M_n(X)$ is also a Banach space.

\begin{example}
  The vector space of complex numbers $\mathbb C$ has a unique operator space
  structure \cite[Chapter 3]{pisier} given by $M_n(\mathbb C) \eqdef M_n$ (see
  paragraph above Definition \ref{def:operator-space}).
\end{example}

Let $V$ be a vector space and let $X$ be a Banach space for which there is a
linear isomorphism $\varphi:V\to X$. Then $V$ can be equipped with a norm
$\| \cdot \|$ defined by $\|v\|\eqdef\|\varphi(v)\|$, and when equipped with this
norm, $V$ becomes a Banach space that is isometrically isomorphic to $X$ via
$\varphi$. If, moreover, $X$ is an operator space, then $V$ can be equipped
with an OSS by defining for each $n\in\mathbb N$ the norm $\|\cdot \|_n$ on
$\MM_n(V)$ by $\|[x_{ij}]\|_n\eqdef\|[\varphi(x_{ij})]\|_n$ for each
$[x_{ij}]\in \MM_n(V)$. With this OSS, $V$ clearly becomes completely
isometrically isomorphic to $X$ via $\varphi$.

\begin{example}
  For an operator space $X$, each of the Banach spaces $M_n(X)$ have a
  canonical operator space structure as well. In particular, we can define a norm on $\MM_m(M_n(X))$ through the linear isomorphism $\MM_m(M_n(X)) \cong M_{mn}(X)$
  and these norms determine an OSS on $M_n(X).$ This gives the canonical OSS on $M_n = M_n(\mathbb C).$
\end{example}

\begin{example}
  Every von Neumann algebra, and more generally, every C*-algebra has a
  canonical OSS. If $A$ is a C*-algebra (von Neumann algebra), then each matrix
  space $\MM_n(A)$ has a unique norm under which it can be equipped with the
  structure of a C*-algebra (von Neumann algebra). These norms then give $A$
  its canonical OSS. Therefore, if $H$ is a Hilbert space, then $B(H)$ has a
  canonical OSS where the norm on $\MM_n(B(H))$ can be defined through the
  linear isomorphism $\MM_n(B(H)) \cong B(H^{\oplus n})$, where the latter
  Banach space is equipped with the operator norm and where $H^{\oplus n}$ is
  the Hilbert space given by the $n$-fold direct sum of $H$ with itself.
  In particular, if $H$ has dimension $n$, then $H \cong \mathbb C^n$ and so the OSS of
  $B(H) \cong M_n$ can be identified with that of $M_n.$
\end{example}

Next, we introduce some of the types of morphisms of operator spaces that we
are interested in.

\begin{definition}
  Let $u \colon X\to Y$ be a linear map between vector spaces $X$ and $Y$. We write
  $u_n:\MM_n(X)\to\MM_n(Y)$ for the linear map  $[x_{ij}]\mapsto [u(x_{ij})]$,
  i.e. the map defined by component-wise application of $u$. If
  $X$ and $Y$ are operator spaces, we say that $u$ is \emph{completely bounded} if
  $\|u\|_{\mathrm{cb}}\eqdef\sup_{n\in\mathbb N}\|u_n\|<\infty$, where $\norm{u_n}$
  is the operator norm of the map $u_n \colon M_n(X) \to M_n(Y)$ between
  the indicated Banach spaces.
  We write $\CB(X,Y)$ for the Banach
  space of all completely bounded maps from $X$ to $Y$ equipped with the
  $\| \cdot \|_{\mathrm{cb}}$-norm. Furthermore, we say that a linear map $u \colon X\to Y$
  between operator spaces is:
  \begin{itemize}
    \item a \emph{complete contraction}, if $u_n$ is a contraction for each $n \in \mathbb N$, equivalently if $\norm{u}_{\mathrm{cb}} \leq 1$;
    \item a \emph{complete isometry}, if $u_n$ is an isometry for each $n \in \mathbb N$;
    \item a \emph{complete quotient map},\footnote{These maps are also known under the name \emph{complete metric surjections}.}
      if $u_n$ is a quotient map for each $n \in \mathbb N$, i.e. if every map $u_n \colon M_n(X) \to M_n(Y)$ maps the open unit ball of $M_n(X)$ \emph{onto} the open unit ball of $M_n(Y)$.
    \item a \emph{completely isometric isomorphism}, if $u$ is a surjective complete isometry.
  \end{itemize}
\end{definition}

\begin{example}\cite[Section 3.2]{effros-ruan}
    Let $X$ and $Y$ be operator spaces. Then $\CB(X,Y)$ has a canonical
    operator space structure that can be defined via the linear isomorphism $
    \MM_n( \CB(X,Y) )\cong \CB(X,M_n(Y))  $ and the Banach space structure of
    the latter space.
\end{example}

A special case of the previous example is obtained if $Y=\mathbb C$. Since for
any commutative C*-algebra $A$ and any bounded linear map $\varphi:X\to A$ we
have $\|\varphi\|_{\mathrm{cb}}=\|\varphi\|$ \cite[Proposition 2.2.6]{effros-ruan}, it
follows that $\CB(X,\mathbb C)=B(X,\mathbb C)=X^*$, where $X^*$ denotes the
Banach space dual of $X$. Hence we have a linear isomorphism $\MM_n(X^*)\cong
\CB(X,M_n)$ \cite[p. 41]{effros-ruan}, which defines a norm on $\MM_n(X^*)$.
Denoting the corresponding normed space by $M_n(X^*)$, we obtain an isometric
isomorphism $M_n(X^*) \cong \CB(X,M_n)$ which yields an operator space structure
on $X^*$.

\begin{example}
  The matrix space $\MM_n(\mathbb C)$ can be equipped with another OSS that is
  called the \emph{trace class} OSS. This can be achieved through the linear
  isomorphism $\MM_n(\mathbb C) \cong M_n^*$ and the OSS of the
  latter space. We write $T_n$ for the corresponding operator space.
\end{example}

\subsection{Limits and Colimits in $\OS$}
\label{sub:os-colimits}

A necessary condition for a category to be locally presentable is for it to be
cocomplete and indeed this is part of the definition that we presented.
Another necessary condition is for it to also be complete \cite[Corollary
1.28]{lp-categories-book} which follows from Definition
\ref{def:locally-presentable} in a non-trivial way.  Despite this, it can be
useful to know how limits are constructed in the category of operator spaces,
so we show how this can be achieved for products and equalisers from which all
other limits can be built in a canonical way.

\begin{definition}
  We write $\OS$ for the (locally small) category whose objects are the
  operator spaces and whose morphisms are the linear complete contractions
  between them.
\end{definition}

It is easy to see that $f \colon X \to Y$ in $\OS$ is an isomorphism in the
categorical sense iff $f$ is a completely isometric isomorphism of operator
spaces. The category $\OS$ has a zero object, written $0$, i.e. an object which
is both initial and terminal, which is given by the zero-dimensional operator
space. The constant-zero map $0_{X,Y} \colon X \to Y$ coincides precisely with
the concept of a zero morphism in $\OS$, i.e. the unique morphism from $X$ to
$Y$ which factorises through the zero object $0$. We sometimes simply write $0$
instead of $0_{X,Y}$ as this usually does not lead to confusion.

\begin{remark}
  The adjunctions \eqref{eq:triple-adjunctions} from Subsection
  \ref{sub:generators} show that the forgetful functor $U \colon \OS \to \Ban$
  is both a left and a right adjoint, so it preserves limits and colimits.
  Therefore limits and colimits in $\OS$ have to be constructed in a compatible
  way to those in $\Ban.$
\end{remark}

Before we may describe the construction of categorical products,
recall that, given an indexed family $(X_\lambda)_{\lambda\in\Lambda}$ of Banach 
spaces, their $\ell^\infty$-direct sum is the Banach space
\[
  \bigoplus^\infty_{\lambda\in\Lambda}X_\lambda \eqdef \left\{ 
(x_\lambda)_{\lambda \in \Lambda} \in \prod_{\lambda\in\Lambda}X_\lambda   \ |\ 
  \sup_{\lambda\in\Lambda}\norm{x_\lambda} < \infty \right\}
\]
equipped with the $\ell^\infty$-norm given by $\norm{(x_\lambda)_{\lambda\in\Lambda}}_\infty\eqdef\sup_{\lambda\in\Lambda}\norm{x_\lambda}$.

\begin{proposition}
  \label{prop:products}
  Let $(X_\lambda)_{\lambda\in\Lambda}$ be an indexed family of operator spaces. Then $\bigoplus^\infty_{\lambda\in\Lambda}X_\lambda$ can be equipped with an OSS such that 
  \begin{itemize}
    \item[(a)] $M_n\left(\bigoplus_{\lambda\in\Lambda}^\infty X_\lambda\right)=\bigoplus^\infty_{\lambda\in \Lambda}M_n(X_\lambda)$.
    \item[(b)] for each $\kappa\in\Lambda$, the canonical inclusion $\iota_\kappa:X_\kappa\to \bigoplus^\infty_{\lambda\in\Lambda}X_\lambda$ is a complete isometry;
    \item[(c)] for each $\kappa\in \Lambda$, the canonical projection $\pi_\kappa \colon \bigoplus^\infty_{\lambda\in \Lambda}X_\lambda \to X_\kappa$ is a complete quotient map;
    \item[(d)] the operator space $\bigoplus^\infty_{\lambda\in\Lambda}X_\lambda$ together with the projections $\pi_\lambda$ constitute the categorical product of the family $(X_\lambda)_{\lambda \in \Lambda}.$
  \end{itemize}
\end{proposition}
\begin{proof}
  This follows immediately from \cite[(1.2.17)]{blecher-merdy}. In slightly more detail,
  property (a) can be seen as the definition of the OSS and (b) and (c) are claimed 
  in \cite[(1.2.17)]{blecher-merdy}. For property (d), if we are given an
  operator space $Z$ together with complete contractions $f_\lambda \colon Z
  \to X_\lambda$ for every $\lambda \in \Lambda$,
  the function $f \colon Z \to
  \bigoplus^\infty_{\lambda\in\Lambda}X_\lambda $ defined by
  $f(z) \eqdef (f_\lambda(z))_{\lambda \in \Lambda}$
  is a complete contraction \cite[(1.2.17)]{blecher-merdy}
  and it is obvious that this is the unique function with the property that
  $\pi_\lambda \circ f = f_\lambda$ for each $\lambda \in \Lambda.$
\end{proof}

For the coproduct construction, we first need to recall the definition of
potentially uncountable sums in a Banach space.
Let $\Lambda$ be a (potentially uncountable) index set. We say that the sum of an
indexed family $(x_\lambda)_{\lambda\in\Lambda}$ of vectors in a Banach space $X$
\emph{converges} if there is $x\in X$ such that the net $\left\{\sum_{\lambda\in F}x_\lambda:F\subseteq\Lambda\text{ finite}\right\}$ converges to $x$,
in which case we write $\sum_{\lambda\in\Lambda}x_\lambda=x$.
The convergence of the net to $x$ implies that for each
$\epsilon>0$ there is a finite subset $F\subseteq \Lambda$ such that for each
finite subset $F \subseteq G\subseteq \Lambda,$ we have
$\left\|\sum_{\lambda \in G} x_\lambda-x\right \| < \epsilon$. We also recall that a sufficient
condition for $\sum_{\lambda \in \Lambda} x_\lambda$ to exist is that $\sum_{\lambda \in \Lambda} \norm{ x_\lambda } < \infty$ and in this case
$\norm{\sum_{\lambda \in \Lambda} x_\lambda} \leq \sum_{\lambda \in \Lambda} \norm{ x_\lambda } . $

Recall that, given an indexed family $(X_\lambda)_{\lambda\in\Lambda}$ of Banach spaces,
their $\ell^1$-direct sum is the Banach space
\[ 
    \bigoplus_{\lambda\in\Lambda}^1X_\lambda \eqdef \left\{ (x_\lambda)_{\lambda\in\Lambda} \in \prod_{\lambda\in\Lambda}X_\lambda \ |\ \sum_{\lambda\in\Lambda}\norm{x_\lambda}<\infty  \right\}
\]
equipped with the $\ell^1$-norm given by
$\norm{(x_\lambda)_{\lambda\in\Lambda}}_1\eqdef\sum_{\lambda\in\Lambda}\norm{x_\lambda}$.

\begin{proposition}
  \label{prop:coproducts}
  Let $(X_\lambda)_{\lambda\in\Lambda}$ be an indexed family of operator spaces. Then $\bigoplus_{\lambda\in\Lambda}^1X_\lambda$ can be equipped with an OSS such that 
  \begin{itemize}
    \item[(a)] there exists a complete isometric isomorphism of dual operator spaces $\left(\bigoplus^1_{\lambda\in\Lambda}X_\lambda\right)^*\cong\bigoplus_{\lambda\in\Lambda}^\infty X_\lambda^*$;
    \item[(b)] for each $\kappa\in\Lambda$, the canonical inclusion $\iota_\kappa:X_\kappa\to \bigoplus^1_{\lambda\in\Lambda}X_\lambda$ is a complete isometry;
    \item[(c)] for each $\kappa\in \Lambda$, the canonical projection $\pi_\kappa:\bigoplus^1_{\lambda\in \Lambda}X_\lambda \to X_\kappa$ is a complete quotient map;
    \item[(d)] for each operator space $Y$, there is a complete isometric isomorphism $\CB\left(\bigoplus_{\lambda\in\Lambda}^1X_\lambda,Y\right)\cong\bigoplus^\infty_{\lambda\in\Lambda}\CB(X_\lambda,Y).$
    \item[(e)] the operator space $\bigoplus^1_{\lambda\in\Lambda}X_\lambda$ together with the inclusions $\iota_\lambda$ constitute the categorical coproduct of the family $(X_\lambda)_{\lambda \in \Lambda}.$
  \end{itemize}
\end{proposition}
\begin{proof}
  This follows easily from \cite[1.4.13]{blecher-merdy}, but we spell out the
  proof in more detail. Property (a) determines the OSS on
  $\bigoplus_{\lambda\in\Lambda}^1X_\lambda$ by using the predual isomorphism. Properties (b), (c) and (d) are also claimed there, so this leaves us with (e) which follows from (d) using some simple arguments.

  If $Z$ is another operator space and $u_\lambda \colon X_\lambda \to Z$, for $\lambda \in \Lambda$, are a family of complete contractions, we have to show that there exists
  a \emph{unique} complete contraction $u \colon \bigoplus_{\lambda \in \Lambda}^{1} X_\lambda \to Z$ with the property that $u \circ \iota_\lambda = u_\lambda.$ 
  The existence of a complete contraction with the required property is shown in \cite[1.4.13]{blecher-merdy}, but uniqueness has not been claimed. To finish the proof, it suffices to show the following (stronger) claim:
  there exists a unique (not necessarily complete) contraction $u \colon \bigoplus_{\lambda \in \Lambda}^{1} X_\lambda \to Z$ with the property that $u \circ \iota_\lambda = u_\lambda.$
  But this is true, because $\bigoplus_{\lambda \in \Lambda}^{1} X_\lambda$ is exactly the construction of the coproduct in $\Ban$.

  The preceding paragraph finished the proof, but nevertheless, let us recall the construction of the unique couniversal map (in $\Ban$).
  Let $X \eqdef \bigoplus_{\lambda \in \Lambda}^{1} X_\lambda$ and let $x \in X.$
  By construction of $X$, we have $x=(x_\lambda)_{\lambda\in\Lambda}$ for $x_\lambda\in X_\lambda$, $\lambda\in\Lambda$, and $\pi_\lambda(x)=x_\lambda$.
  By definition of the norm on $X$, we have $\|x\|_1=\sum_{\lambda\in\Lambda}\|x_\lambda\|$.
  Let $\epsilon>0$.
  Then there is a finite $F\subseteq \Lambda$ such that for each finite $G\subseteq \Lambda$ with $F\subseteq G$, we have $\left|\sum_{\lambda\in G}\|x_\lambda\|-\|x\|_1\right|<\epsilon$.
  Therefore,
  \[ 
    \left|\sum_{\lambda\in G}\|x_\lambda\|-\|x\|_1\right|
    = \left| -1 \cdot \sum_{\lambda \in (\Lambda -G)} \|x_\lambda\|\right| = \sum_{\lambda \in (\Lambda - G)}\|x_\lambda\| <\epsilon  . 
  \]
  Then, $\left\|\sum_{\lambda\in G}\iota_\lambda(x_\lambda)-x\right\|_1=\sum_{\lambda\in (\Lambda - G)} \norm{x_\lambda}< \epsilon$,
  which shows that $x=\sum_{\lambda\in\Lambda}\iota_\lambda(x_\lambda)$.    
  Therefore, the unique bounded map with the desired property is given
  by $u \left( \sum_\lambda \iota_\lambda(x_\lambda) \right) \eqdef \sum_\lambda u_\lambda(x_\lambda).$
  To see that the right sums exists (and therefore $u$ is well-defined), we note that $\sum_\lambda \norm{u_\lambda(x_\lambda)}
  \leq \sum_\lambda \norm{ x_\lambda} = \norm{x} < \infty$ and so the sum is absolutely convergent and therefore $\sum_\lambda u_\lambda(x_\lambda)$ converges. To see that $u$ is a contraction, observe that
  $\norm{u(x)} = \norm{\sum_\lambda u_\lambda(x_\lambda)} \leq \sum_\lambda \norm{u_\lambda(x_\lambda)} \leq \sum_\lambda \norm{x_\lambda}  = \norm{x} . $ Uniqueness now follows easily, because if $u' \colon X \to Z$
  is a contraction with $u' \circ \iota_\lambda = u_\lambda$, then for any $x\in X$, we have that
  \[ u'(x) = u'\left( \sum_\lambda \iota_\lambda(x_\lambda) \right) = \sum_\lambda u'(\iota_\lambda(x_\lambda)) = \sum_\lambda u_\lambda(x_\lambda) = u(x) , \]
  where the second equality follows by continuity and linearity of $u'.$ Therefore $u$ is the unique (complete) contraction with the required couniversal property.
\end{proof}

Next, we describe the construction of equalisers which is completely analogous to the case for Banach spaces.

\begin{proposition}
  \label{prop:equalisers}
  Let $f, g \colon X \to Y$ be two complete contractions. Their equaliser is
  given by the closed subspace $E \eqdef \{ x \in X\ |\ f(x) = g(x) \}$
  together with the completely isometric inclusion $E \subseteq X.$
\end{proposition}
\begin{proof}
  $E$ is obviously a closed subspace of $X$ and if we write $e \colon E \to X$ for the subspace inclusion, then clearly $f \circ e = g \circ e.$
  If $Z$ is another operator space and $h \colon Z \to X$ is a complete contraction such that $f \circ h = g \circ h$, then $\Image(h) \subseteq E$,
  so we may corestrict $h$ to $\hat h \colon Z \to E$ and then $\hat h$ is obviously the unique complete contraction with the property $e \circ \hat h = h.$
\end{proof}

In order to construct coequalisers, we first need the concept of a
quotient operator space. We recall that if $N$ is a closed subspace of a
Banach space, then $X/N$ becomes a Banach space if we define
$\norm{[x]}_{X/N}=\inf_{y\in N}\norm{x+y}=\inf_{y\in[x]}\norm{y}$ for each
$[x]\in X/N$.  Hence, the quotient map $q:X\to X/N$ is always a contraction. If
$X$ is an operator space, then $M_n(N)$ is a closed subspace of $M_n(X)$, hence
$N$ is an operator space, which leads to the next proposition.

\begin{proposition}[{\cite[Proposition 3.1.1]{effros-ruan}}]
    Let $X$ be an operator space and $N\subseteq X$ a closed subspace. Then there is an OSS on $X/N$ such that $M_n(X/N)=M_n(X)/M_n(N)$.
    Moreover, the quotient map $q:X\to X/N$ is a \emph{complete quotient map}.
\end{proposition}

\begin{proposition}
  \label{prop:coequalisers}
  Let $f,g \colon X \to Y$ be two complete contractions. Then their coequaliser is given by the operator space $E \eqdef Y / \overline{\Image(f-g)}$ together with the complete quotient map $q \colon Y \to E :: y \mapsto [y].$
\end{proposition}
\begin{proof}
  First, it is easy to see that $q \circ f = q \circ g.$ Indeed, $q \circ f = q
  \circ g$ iff $\forall x \in X.$ $[f(x)] = [g(x)]$ iff $\forall x \in X.$
  $f(x) - g(x) = (f-g)(x) \in \overline{\Image(f-g)}$ which is clearly true.

  To show couniversality, let $h \colon Y \to Z$ be a complete contraction such
  that $h \circ f = h \circ g.$ This implies that $\Image(f-g) \subseteq \ker(h)$ and since $\ker(h)$ is closed, then $\overline{\Image(f-g)} \subseteq \ker(h)$
  Then, the canonical map $\hat h \colon E \to Z :: [y] \mapsto h(y)$ induced by $h$ 
  and the quotient structure is a complete contraction \cite[(1.2.15)]{blecher-merdy}. Obviously, this is also the unique complete contraction such that the diagram
  \cstikz{coequaliser-proof.tikz}
  commutes.
\end{proof}

\begin{proposition}
  \label{prop:cocomplete}
  The category $\OS$ is complete and cocomplete.
\end{proposition}
\begin{proof}
  Any category with coequalisers (Proposition \ref{prop:coequalisers}) and
  small coproducts (Proposition \ref{prop:coproducts}) has all small colimits.
  The existence of small limits follows from Proposition \ref{prop:equalisers}
  and Proposition \ref{prop:products}.
\end{proof}

\subsection{Strong generators in $\OS$}
\label{sub:generators}

For the locally presentable structure of a category, it is useful to
identify its strong generators. This is the main purpose of
this subsection. As a first step, it is useful to classify the monomorphisms
and the epimorphisms. The situation is completely analogous to the
classification of the same notions in $\Ban,$ where monomorphisms correspond to
the contractive injections and where the epimorphisms correspond to the
contractions with dense image.  The proofs are also completely analogous (see
\cite{borceux1}), but we provide them for completeness.

\begin{proposition}
  \label{prop:monomorphisms}
  Let $m \colon X \to Y$ be a complete contraction between operator spaces $X$
  and $Y.$ Then, $m$ is a monomorphism iff $m$ is an injection.
\end{proposition}
\begin{proof}
  $(\Rightarrow)$ Let $x_1, x_2 \in X$ and assume that $m(x_1) = m(x_2)$. Let $r > 0$ be such that
  $rx_1, rx_2 \in \Ball(X).$ We can now construct two complete contractions $f,g \colon \mathbb C \to X$
  such that $f(1) = rx_1$ and $g(1) = rx_2$. Then $m(x_1) = m(x_2)$ iff $m(rx_1) = m(rx_2)$ iff $(m \circ f)(1) = (m \circ g)(1)$ iff $m \circ f = m \circ g$.
  Since $m$ is a monomorphism, it follows $f = g$ and therefore $x_1 = x_2.$

  $(\Leftarrow)$ Obvious.
\end{proof}

\begin{proposition}
  \label{prop:epimorphism}
  Let $e \colon X \to Y$ be a complete contraction between operator spaces $X$
  and $Y.$ Then, $e$ is an epimorphism iff the image $e[X]$ is dense in $Y.$
\end{proposition}
\begin{proof}
  $(\Rightarrow)$ Consider the complete quotient map $q \colon Y \to Y/\overline{e[X]}$
  and the constant zero map $0_Y \colon Y \to Y/\overline{e[X]}$. We have that
  $q \circ e = 0_X = 0_Y \circ e$, where $0_X \colon X \to Y/\overline{e[X]}$
  is the constant zero map. Since $e$ is an epimorphism, then $q = 0_Y$ and
  therefore $\overline{e[X]} = Y$ and thus $e[X]$ is indeed dense in $Y.$

  $(\Leftarrow)$ Let $Z$ be an operator space and let $f,g \colon Y \to Z$ be
  two complete contractions.  If $f \circ e = g \circ e$, then $f$ and $g$
  coincide on $e[X]$ which is dense in $Y$. Since both functions are
  continuous (and $Y$ is Hausdorff), it follows that $f = g.$
\end{proof}

Next, we classify special kinds of monomorphisms and epimorphisms that behave
better in locally presentable categories compared to generic monomorphisms and
epimorphisms. The situation (and proofs) are again completely analogous to the
ones in $\Ban.$

\begin{proposition}
  \label{prop:strong-mono}
  Let $m \colon X \to Y$ be a complete contraction between operator spaces $X$ and $Y$. The following are equivalent:
  \begin{enumerate}
    \item[(1)] $m$ is a regular monomorphism;
    \item[(2)] $m$ is a strong monomorphism;
    \item[(3)] $m$ is an extremal monomorphism;
    \item[(4)] $m$ is a complete isometry.
  \end{enumerate}
\end{proposition}
\begin{proof}
  The implications $(1) \Rightarrow (2) \Rightarrow (3)$ hold in every category, so it suffices to prove
  $(3) \Rightarrow (4)$ and $(4) \Rightarrow (1).$

  $(3) \Rightarrow (4).$
  Let $e \colon X \to \overline{m[X]}$ be the corestriction of $m$, i.e. $e(x) \eqdef m(x).$
  Let $i \colon \overline{m[X]} \to Y$ be the subspace inclusion. By Proposition \ref{prop:epimorphism}, $e$ is an epimorphism
  and we have that $m = i \circ e$. Since $m$ is an extremal monomorphism, it follows that $e$
  is a categorical isomorphism, i.e. a completely isometric isomorphism of operator spaces.
  Since $i$ is a complete isometry, then so is $m = i \circ e$.

  $(4) \Rightarrow (1).$
  Consider the complete quotient map $q \colon Y \to Y/m[X]$ and the zero map
  $0_Y \colon Y \to Y/m[X].$ By using Proposition \ref{prop:equalisers}, we see
  that the equaliser of $q$ and $0_Y$ is given by the subspace inclusion $m[X] \subseteq Y$. But
  $X \cong m[X]$ completely isometrically, therefore $m \colon X \to Y$ is also the equaliser
  of $q$ and $0_Y.$
\end{proof}

\begin{proposition}
  \label{prop:strong-epi}
  Let $e \colon X \to Y$ be a complete contraction between operator spaces $X$ and $Y$. The following are equivalent:
  \begin{enumerate}
    \item[(1)] $e$ is a regular epimorphism;
    \item[(2)] $e$ is a strong epimorphism;
    \item[(3)] $e$ is an extremal epimorphism;
    \item[(4)] $e$ is a complete quotient map.
  \end{enumerate}
\end{proposition}
\begin{proof}
  Again, the implications $(1) \Rightarrow (2) \Rightarrow (3)$ hold in every category, so it suffices to prove
  $(3) \Rightarrow (4)$ and $(4) \Rightarrow (1).$

  $(3) \Rightarrow (4).$ Let $q \colon X \to X/\ker(e)$ be the complete quotient map with the indicated type and let $\hat e \colon X/\ker(e) \to Y$ be
  the canonical complete contraction such that $e = \hat e \circ q.$ The map $\hat e$ is clearly injective and therefore a monomorphism. Since
  $e$ is an extremal epimorphism, it follows that $\hat e$ is a categorical isomorphism, i.e. a completely isometric isomorphism. Therefore
  $e = \hat e \circ q$ is a composition of two complete quotient maps and it is therefore a complete quotient map itself.

  $(4) \Rightarrow (1).$ Since $e$ is a complete quotient map, we know that the
  following diagram
  \cstikz{regular-epi-complete-quotient.tikz}
  commutes and the canonical map $\hat e$ from the preceding argument is a
  complete isometric isomorphism. Since $\hat e$ is a categorical isomorphism,
  it follows that $e$ is a regular epimorphism iff $q$ is one, so we prove the
  latter. Consider the operator space direct sum $X \oplus^\infty X$ and let
  $E \eqdef \{ (x_1, x_2) \ |\ e(x_1) = e(x_2) \}$ be the closed subspace of
  elements which agree under $e$. The projections $\pi_1, \pi_2 \colon E \to X$
  are complete contractions and are known as the \emph{kernel pair} of $e$, and
  we will show that $q$ is their coequaliser. It is easy to see that
  $\mathrm{Im}(\pi_1 - \pi_2) = \ker(e)$. Therefore, $\mathrm{Im}(\pi_1-\pi_2)$ is also
  closed. It now follows from  Proposition \ref{prop:coequalisers} that $q$ is exactly
  the aforementioned coequaliser.
\end{proof}

Next, we recall the following lemma that is crucial for proving that the set
$\{T_n \ |\ n \in \mathbb N\}$ is strongly generating for $\OS.$

\begin{lemma}
  \label{lem:trace-class-contraction}
  Let $X$ be an operator space and $[x_{ij}] \in M_n(X)$ be such that $\norm{[x_{ij}]} \leq 1.$ Then, the linear map
  \[ u \colon T_n \to X :: e_{ij} \mapsto x_{ij} \]
  is a complete contraction. Moreover, by making use of the identification $M_n(T_n) \cong \CB(M_n, M_n)$, we have that
  $u_n(\id_{M_n}) = [x_{ij}].$
\end{lemma}
\begin{proof}
  That $u$ is a complete contraction is shown in {\cite[pp. 101]{effros-ruan}}.
  The second statement follows immediately (as a special case) using the
  arguments that follow on the same page.
\end{proof}

\begin{proposition}
  \label{prop:strong-generator}
  The set $\{ T_n \ |\ n \in \mathbb N \},$ consisting of the finite-dimensional trace class operator spaces, is strongly generating for the category $\OS.$
\end{proposition}
\begin{proof}
  It is easy to see that the complex numbers $\C$ is a generating object.
  Indeed, if $f,g \colon X \to Y$ are two distinct complete contractions, then
  there must exist an element $x \in B_{\leq 1}(X)$ in the unit ball of $X$
  such that $f(x) \neq g(x)$, otherwise a simple normalisation argument yields
  a contradiction. Then, the map $1 \mapsto x : \C \to X$ is a complete
  contraction and this shows that $\C$ is a generating object. Since
  $\C \cong T_1$ as operator spaces, it follows that $\{ T_n \ |\ n \in \mathbb N \}$ is
  also a generating set.

  To show that this family is strongly generating, let $f \colon X \to Y$ be a
  proper monomorphism between two operator spaces. Since $f$ is a monomorphism,
  it is a completely contractive injection. Since it is a \emph{proper}
  monomorphism, it is not a complete quotient map (equivalently, an extremal
  epimorphism).
  Therefore, there exists $n \in \mathbb N$ and $y \in M_n(Y)$ with $\norm{y} < 1$, such that $y \neq f_n(x)$,
  for any $x \in M_n(X)$ with $\norm{x}<1$. Therefore, if $y = f_n(x)$ for some $x \in M_n(X)$, it must be the
  case that $\norm{y} < 1 \leq \norm{x}$. The element $y$ is clearly not $0$, so after normalising $y' \eqdef \frac{y}{\norm{y}},$
  we see that if $f_n(x) = y'$ for some $x \in M_n(X)$, then it must be the case that $1 < \norm{x}.$

  We can now apply Lemma \ref{lem:trace-class-contraction} to $y'$ to
  find a complete contraction $u \colon T_n \to
  Y$ such that $u_n(t) = y'$ for some $t \in M_n(T_n) \cong \CB(M_n, M_n)$
  with $\norm{t} = 1.$ Note that $t$ is determined by $\id_{M_n} \in \CB(M_n, M_n)$
  and the isomorphism $M_n(T_n) \cong \CB(M_n, M_n)$. The contraction $u_n$
  cannot factor via a contraction through $f_n$ due to the argument in the
  preceding paragraph. In other words, if the following diagram of
  \emph{bounded} maps
  \cstikz{factor-bounded.tikz}
  commutes, then it must be the case that $\norm{g(t)} > 1 = \norm t,$ so
  $g$ is not a contraction. Therefore $u$ cannot factor via a complete
  contraction through $f.$ 
\end{proof}

\begin{corollary}
  The operator space $\oplus^1_{n \in \mathbb N} T_n  $ is a strong generator for $\OS$.
\end{corollary}
\begin{proof}
  Let $\pi_k \colon \oplus^1_{i \in \mathbb N} T_i \to T_k $ be the canonical complete quotient map (see Proposition \ref{prop:coproducts}).
  From Proposition \ref{prop:strong-generator}, it follows that if $f, g \colon
  C \to D$ are two distinct complete contractions, then we can find $k \in \mathbb N$ and a
  morphism $t \colon T_k \to C$ such that $f \circ t \neq g \circ t$ and
  therefore $f \circ t \circ \pi_k \neq g \circ t \circ \pi_k$ using the fact
  that $\pi_k$ is a complete quotient map, so an epimorphism. Therefore, $\oplus^1_{n \in \mathbb N} T_n  $ is a generator.
  To prove that it is strong, let $m \colon C \to D$ be a proper monomorphism.
  From Proposition \ref{prop:strong-generator}, we can find $k \in \mathbb N$
  and $f \colon T_k \to D$ such that $f$ does not factorise through $m$, i.e.
  there exists no complete contraction $d \colon T_k \to C$ with the property that
  $f = m \circ d$. Since $\pi_k$ is a strong epimorphism (Proposition \ref{prop:strong-epi}), it follows that $f \circ \pi_k$ cannot factorise
  through $m$, because otherwise we would be able to construct the
  aforementioned map $d$ by definition of strong epimorphism. 
\end{proof}

\begin{corollary}
  \label{cor:trace-class-generator}
  The operator space $T(\ell_2)$ is a strong generator for $\OS.$
\end{corollary}
\begin{proof}
  We can construct a complete quotient map $q \colon T(\ell_2) \to \oplus^1_{n
  \in \mathbb N} T_n $ (see \cite[pp. 68]{pisier}). The proof follows from the
  preceding corollary by using essentially the same arguments. 
\end{proof}

Now that we have identified several strong generators, it can be useful to
understand why the situation is different compared to $\Ban$, where the complex
numbers $\C$ are already a strong generator. The reason is similar to the
reason why the terminal object $1$ is a non-strong generator in $\Top.$ There
exists an adjoint situation
  \cstikz{triple-adjunction-top.tikz}
where $D$ is the functor that assigns the discrete topology to a set and $I$ is the functor that assigns the indiscrete one.
Then, if we consider the following diagram in $\Top$,
  \cstikz{non-strong-top.tikz}
we see that the identity map $\id \colon D(\mathbb N) \to I(\mathbb N)$ is a proper monomorphism and it is obvious that for any map $f \colon 1 \to I(\mathbb N),$ the above diagram commutes,
so $1$ is not a strong generator. The proof in $\OS$ can be shown in a similar way.

Let $U \colon \OS \to \Ban$ be the obvious forgetful functor.
Let $\Max \colon \Ban \to \OS$ be the maximal quantisation functor, i.e. the
functor that assigns to a Banach space $X$ the biggest possible OSS in the
sense that $(U \circ \Max)(X) = X$ and such that for each $n \geq 2,$
the norm $\norm{\cdot}_n$ on $M_n(X)$ is the largest norm among all such OSS norms (see
\cite[\S 3.3]{effros-ruan} for more information). We write $\Min \colon \Ban
\to \OS$ for the minimal quantisation functor, which has the property that
$(U \circ \Min)(X) = X$ and which assigns the smallest possible OSS
that is compatible with $X$ (again, see \cite[\S 3.3]{effros-ruan} for more
information). It has already been observed by Yemon Choi that this gives us
an adjoint situation similar to the one in $\Top$ above
\cite{choi-adjunctions}, 
\begin{equation}
  \label{eq:triple-adjunctions}
  \stikz{triple-adjunction-os.tikz}
\end{equation}
see also \cite{pestov-adjunction} where a similar adjunction is used.

\begin{proposition}
  The operator space $\mathbb C$ is a non-strong generator for $\OS$.
\end{proposition}
\begin{proof}
  We already proved that $\C$ is a generator in Proposition
  \ref{prop:strong-generator}, so we just have to show that it is not strong.
  Consider   the following diagram in $\OS,$
  \cstikz{non-strong-os.tikz}
  where $X = B(\ell_2)$ viewed as a Banach space. In fact, for $X$ we can choose any Banach space for which $\Max(X) \not \cong \Min(X)$ as operator spaces.
  Then, the identity map $\id \colon \Max(X) \to \Min(X)$ is a completely contractive injection which is not completely isometric, so it is a proper monomorphism.
  Also, $f \colon \mathbb C \to \Min(X)$ is a contraction iff it is a complete contraction and
  $f \colon \mathbb C \to \Max(X)$ is a complete contraction iff it is a contraction, because $\mathbb C = \Max(\mathbb C).$
  Since $U(\Min(X)) = U(\Max(X)) = X$, it follows that the diagram above is
  well-defined in $\OS$ and it clearly commutes for any choice of $f \colon \C
  \to \Min(X)$. Therefore $\C$ cannot be a strong generator.
\end{proof}

\subsection{Countably-presentable objects in $\OS$}
\label{sub:presentable}

Before we classify the countably-presentable objects in $\OS$, let us first show that the finitely-presentable ones are trivial.
The argument is completely analogous to the one we gave for $\Ban$ in Section \ref{sec:banach}, but we give the proof in more detail.

\begin{proposition}
  The only finitely-presentable object in $\OS$ is $0.$
\end{proposition}
\begin{proof}
  It is obvious that $0$ is finitely-presentable, so let us assume that $X$ is an operator space with dimension greater than zero. 
  Consider the commutative von Neumann algebra $\ell_\infty$ viewed as an operator space. For each $i \in \mathbb N$, consider the closed subspace
  \[ D_i \eqdef \{ (a_k)_{k \in \mathbb N} \ |\ a_k = 0, \text{ for every } k > i \} \subset \ell_\infty . \]
  Viewing each $D_i$ as an operator space, we have completely isometric
  inclusions $D_i \subset D_j$ for $i < j$, so we obtain a directed diagram $D$
  in $\OS$. Each of the operator spaces $D_i$ and $\ell_\infty$ are minimal
  operator spaces, because they are subspaces of a commutative C*-algebra
  \cite[Proposition 3.3.1]{effros-ruan}.  We have that
  \[ c_{00} = \bigcup_{i \in \mathbb N} D_i \]
  as normed spaces. But the colimit of $D$ in $\OS$ is the operator space
  determined by the closure of $c_{00}$ in $\ell_\infty$ (see \cite[pp. 37]{effros-ruan}
  for more information), equivalently the completion of $c_{00},$ i.e.
  \[ \colim(D) = \Min(c_0) \subset \ell_\infty . \] 
  Clearly, $c_{00} \subsetneq \Min(c_0)$, so we can find a linear contraction
  $f \colon X \to \Min(c_0)$ whose image $f[X]$ contains an element that is not
  in $c_{00}$, e.g. by using the Hahn-Banach theorem.
  Any such contraction is necessarily a complete contraction (see
  also the adjoint situation in \eqref{eq:triple-adjunctions} for a categorical
  justification) and it is clear that it cannot factorise through any of the
  inclusions $D_i \subset \Min(c_0)$. Therefore, $X$ is not finitely-presentable.
\end{proof}

In order to determine the countably-presentable objects of $\OS$, we use an
explicit description of colimits of countably-directed diagrams in $\OS$, which
turns out to be very similar to the construction in $\Ban$.

\begin{construction}[Countably-directed colimits in $\mathbf{OS}$]\label{constr:aleph1-directed-colimit}
  Let $D:\Lambda\to\mathbf{OS}$ be a countably-directed diagram in $\mathbf{OS}$, and let $(C, \{ c_\lambda\}_{\lambda \in \Lambda})$ be its colimiting cocone
  in $\Ban$ as in Construction \ref{constr:aleph1-directed-colimit-Ban}.
  We define an operator space structure on $C$ by complete analogy:
  if $v \in \mathbb M_n(C)$, let
  \[ \norm{ v } \eqdef \inf\ \{ \norm{w} \ |\ \kappa \in \Lambda, w \in M_n(D_\kappa), c_\kappa^n(w) = v \} , \]
  where $c_\kappa^n \eqdef (c_\kappa)_n \colon M_n(D_\kappa) \to \mathbb M_n(C)$ is simply the $n$-th amplification of $c_\kappa.$
  We write $M_n(C)$ for the space $\MM_n(C)$ equipped with the norm defined above.
\end{construction}

We have to justify that $M_n(C)$ from the above construction is a Banach space.
This is true because $M_n(C)$ can be recovered as a suitable colimit in $\Ban$,
as we show next.

\begin{lemma}\label{lem:amplification-colimits}
    Let $D \colon \Lambda \to \OS$ be a countably-directed diagram in $\OS$.
    For each $n\in\mathbb N$, let $D_n:\Lambda\to\mathbf{Ban}$ be the diagram
    given by  $D_n(\lambda)\eqdef M_n(D_\lambda)$ and
    $D_n(\lambda\leq\kappa)\eqdef (D_{\lambda,\kappa})_n$. For simplicity, we
    write $D^n_{\lambda,\kappa}$ instead of $(D_{\lambda,\kappa})_n$. Then
    $M_n(C)$ as described in Construction \ref{constr:aleph1-directed-colimit}
    is the colimit of $D_n$ in $\Ban$ with colimiting maps
    $c_\lambda^n:M_n(D_\lambda)\to M_n(C)$. 
\end{lemma}
\begin{proof}
    It is well known that the category $\mathbf{Vect}$ of vector spaces over
    $\mathbb C$ is symmetric monoidal closed with respect to the algebraic
    tensor product $\otimes$, where the internal hom $[X,Y]$ between two vector
    spaces is given by the set of linear maps between them. As a consequence,
    the endofunctor $\MM_n \otimes(-) \colon \mathbf{Vect} \to \mathbf{Vect}$ preserves colimits.
    Note that this functor is naturally isomorphic to the functor $\MM_n(-) \colon \mathbf{Vect} \to \mathbf{Vect}$
    defined on objects by $X\mapsto \MM_n(X)$ and on
    linear maps by $u \mapsto u_n$. Since $C$ is the colimit of $D$ in
    $\mathbf{Vect}$, it follows that the linear maps
    $c_\lambda^n:\MM_n(D_\lambda)\to\MM_n(C)$ form the colimiting cocone of
    the diagram $D_n$ regarded as a functor $\Lambda\to\mathbf{Vect}$.
    Taking the norms of the objects in the diagram into account, i.e. the Banach spaces $M_n(D_\lambda)$,
    we see that the norm on $\MM_n(C)$ from Construction
    \ref{constr:aleph1-directed-colimit-Ban} (in $\Ban)$ coincides with the norm on $\MM_n(C)$
    that we just defined in Construction \ref{constr:aleph1-directed-colimit}.
    Therefore, using Corollary \ref{cor:colimits-ban}, we see that the colimit
    of $D_n:\Lambda\to\mathbf{Ban}$ is $M_n(C)$ together with the colimiting
    maps $c_\lambda^n.$
\end{proof}

We can now show that the construction indeed yields an operator space.

\begin{proposition}
  \label{prop:construction-operator-space}
  The norms defined in Construction \ref{constr:aleph1-directed-colimit} give $C$ an operator space structure.
\end{proposition}
\begin{proof}
  We already know that $C$ is Banach space, so it remains to show that
  $C$ satisfies axioms (M1) and (M2) of abstract operator spaces.
  In order to verify (M1), let $x\in M_m(C)$ and $y\in M_n(C)$, then
  \begin{align*}
      \norm{x}_m & = \inf\ \{ \norm{v}_m \ |\ \kappa \in \Lambda, v \in M_m(D_\kappa), c_\kappa^m(v) = x \}\\
      \norm{y}_n &= \inf\ \{ \norm{w}_n \ |\ \kappa \in \Lambda, w \in M_n(D_\kappa), c_\kappa^n(w) = y \} \\
      \norm{x\oplus y}_{n+m} & = \inf\ \{ \norm{z}_{n+m} \ |\ \kappa \in \Lambda, z \in M_{n+m}(D_\kappa), c_\kappa^{n+m}(z) = x\oplus y \}.
  \end{align*}
  Consider the one-dimensional closed subspaces spanned by $x$, $y$ and
  $x \oplus y$ in $M_m(C), M_n(C)$ and $M_{n+m}(C)$, respectively.
  Combining Lemma \ref{lem:amplification-colimits} and
  Proposition \ref{prop:closed-separable}, we can find
  $\lambda_m, \lambda_n$ and $\lambda_{n+m}$ in $\Lambda$,
  such that $c_{\lambda_m}^m, c_{\lambda_n}^n$ and $c_{\lambda_{n+m}}^{n+m}$
  (co)restrict to isometric isomorphisms with respect to these subspaces.
  Since the diagram $D$ is directed, we can find
  $\tau \in \Lambda$, such that $\tau$ is an upper bound of
  $\{\lambda_m, \lambda_n, \lambda_{n+m}\}$ in $\Lambda$.
  Let $v \in M_m(D_\tau), w \in M_n(D_\tau),$ and $t \in M_{n+m}(D_\tau)$ be the unique elements in the corresponding subspaces such that $c_\tau^m(v) = x,
  c_\tau^n(w) = y,$ and $c_\tau^{n+m}(t) = x \oplus y$. It follows that $\norm v_m = \norm x_m,\norm w_n = \norm y_n,$ and $\norm t_{n+m} = \norm{x \oplus y}_{n+m}.$
  We obviously have
  \[ c_\tau^{n+m} (v \oplus w) = c_\tau^m(v) \oplus c_\tau^n(w) = x \oplus y = c_\tau^{n+m}(t), \]
  however we do not know if $t = v \oplus w$. To resolve this,
  by construction of the colimit in $\Ban$, the above equation implies that there exists $\lambda \geq \tau$ such that
  $D_{\tau, \lambda}^{m+n}(v \oplus w) = D_{\tau, \lambda}^{m+n}(t).$
  It again follows that
  \begin{align*}
    c_\lambda^m &\colon M_m(D_\lambda) \to M_m(C) \\
    c_\lambda^n &\colon M_n(D_\lambda) \to M_n(C) \\
    c_\lambda^{n+m} &\colon M_{n+m}(D_\lambda) \to M_{n+m}(C)
  \end{align*}
  (co)restrict to isometric isomorphisms with respect to the aforementioned one-dimensional subspaces.
  After defining
  \[ v' \eqdef D_{\tau,\lambda}^m(v), w' \eqdef D_{\tau,\lambda}^n(w), \text{ and } t' \eqdef D_{\tau,\lambda}^{n+m}(t) , \]
  we now have
  \[ \norm{v'}_m = \norm{x}_m,\norm{w'}_n = \norm{y}_n, \text{ and } \norm{t'}_{n+m} = \norm{x \oplus y}_{n+m}, \]
  but also $t' = v' \oplus w',$ so that $\norm{v' \oplus w'}_{n+m} = \norm{x \oplus y}_{n+m}.$
  Therefore
  \begin{align*}
    \norm{x \oplus y}_{n+m} &= \norm{v' \oplus w'}_{n+m} & \\
                            &= \mathrm{max} \{ \norm{v'}_m , \norm{w'}_n \} & (\text{axiom (M1) for } D_\lambda) \\
                            &= \mathrm{max} \{ \norm x_m , \norm y_n \} & 
  \end{align*}
  which shows that axiom (M1) holds for $C$.

  For (M2), let $m\in\mathbb N$, $x\in M_m(C)$, and $\alpha,\beta\in M_m$.
  Using essentially the same arguments as above, we can find $\tau \in \Lambda$
  such that $c_\tau^m \colon M_m(D_\tau) \to M_m(C)$ (co)restricts to an
  isometric isomorphism on the closed subspace
  \[ \mathrm{span}(x, \alpha x \beta) \subseteq  M_m(C). \]
  Let $v \in M_m(D_\tau)$ and $t \in M_m(D_\tau)$ be the unique elements in the corresponding subspace such that $c_\tau^m(v) = x$ and $c_\tau^m(t) = \alpha x \beta$.
  The linearity of $c_\tau^m$ implies that
  \[ c_\tau^m(\alpha v \beta) = \alpha c_\tau^m(v) \beta = \alpha x \beta = c_\tau^m(t). \]
  Again, we can now find $\lambda \geq \tau$ such that $D_{\tau, \lambda}^m(t) = D_{\tau, \lambda}^m(\alpha v \beta)$ and after defining
  \[ v' \eqdef D_{\tau, \lambda}^m(v) \text{ and } t' \eqdef D_{\tau, \lambda}^m(t) \]
  we have
  \[ \alpha v' \beta = \alpha D_{\tau, \lambda}^m(v) \beta =  D_{\tau, \lambda}^m(\alpha v \beta ) = D_{\tau, \lambda}^m(t) = t' . \]
  It follows that $\norm{v'}_m = \norm{x}_m$ and $\norm{\alpha v' \beta}_m = \norm{t'}_m = \norm{\alpha x \beta}_m.$
  By axiom (M2) for $D_\lambda$ we have that $\norm{\alpha v' \beta}_m \leq \norm{\alpha} \norm{v'}_m \norm{\beta}$
  and therefore
  \[ \norm{\alpha x \beta}_m = \norm{\alpha v' \beta}_m \leq \norm{\alpha} \norm{v'}_m \norm{\beta} = \norm \alpha \norm{x}_m \norm \beta \]
  which shows that axiom (M2) holds for $C.$
  Therefore $C$ is indeed an operator space.
\end{proof}

\begin{proposition}
  \label{prop:aleph1-directed-colimit}
  Let $D \colon \Lambda \to \OS$ be a countably-directed diagram in $\OS$. Then the colimit of $D$ is given by Construction \ref{constr:aleph1-directed-colimit}.
\end{proposition}
\begin{proof}
  We just showed that $C$ is an operator space and since $C$ is the colimit of $D$ in $\mathbf{Ban}$, it is sufficient to verify that:
  \begin{itemize}
      \item[(1)] $c_\lambda$ is a complete contraction for each $\lambda\in \Lambda$;
      \item[(2)] if $X$ is another operator space, and $d_\lambda: D_\lambda\to X$ forms a cocone, then there is a unique complete contraction $f:C\to X$ such that $f\circ c_\lambda=d_\lambda$ for each $\lambda\in \Lambda$. 
  \end{itemize}
  We start with (1). Let $\lambda\in\Lambda$, let $n\in\mathbb N$, and let $v\in M_n(D_\lambda)$. Then, by construction of the norm on $M_n(C)$, we have
  \begin{align*}
  \norm{c_\lambda^n(v)}_n & = \inf\{\norm{w}_n\ |\ \kappa\in\Lambda,w\in M_n(D_\kappa), c_\kappa^n(w)=c_\lambda^n(v)\}\leq\norm{v}_n,
  \end{align*}
  so $c^n_\lambda$ is a contraction for each $n\in\mathbb N$. Therefore $c_\lambda$ is a complete contraction.

  For (2), given a cocone $d_\lambda: D_\lambda\to C$
  in $\mathbf{OS}$, it follows that $d_\lambda$ also forms a cocone in
  $\mathbf{Ban}$. Fix $n\in\mathbb N$, and let
  $d_\lambda^n:M_n(D_\lambda)\to M_n(C)$ denote the $n$-th amplification of
  $d_\lambda$. Then the maps $d_\lambda^n$ form a cocone of the diagram $D_n$ in $\mathbf{Ban}$
  and by Lemma \ref{lem:amplification-colimits} it follows that
  there is a unique contraction $g^{(n)}:M_n(C)\to M_n(X)$ such that
  $g^{(n)}\circ c_\lambda^n=d_\lambda^n$ for each $\lambda\in\Lambda$.
  Let $f= g^{(1)}$. Then
  also the $n$-th amplification $f_n$ of $f$ satisfies $f_n\circ c_\lambda^n=d_\lambda^n$.
  But $M_n(C)$ is also the colimit of $D_n$ as a diagram in $\Vect$, so it
  follows that $f_n=g^{(n)}$. Therefore $f_n$ is a contraction for each
  $n\in\mathbb N$ and so $f$ is a complete contraction.
\end{proof}

Recall that an operator space $X$ is \emph{separable} whenever $X$ is separable
as a Banach space. Next, we want to show that any separable operator space is
countably-presentable in $\OS.$ First, we need an additional lemma.

\begin{lemma}[{\cite[pp. 20]{effros-ruan}}]
  \label{lem:norm X vs norm MnX}
  Let $X$ be an operator space. Then for each $n\in\mathbb N$ and each $x=[x_{ij}]$ in $M_n(X)$, we have
    \begin{itemize}
        \item[(1)] $\|x_{ij}\|\leq \|x\|$ for each $i,j\in\{1,\ldots,n\}$;
        \item[(2)] $\|x\|\leq\sum_{i,j=1}^n\|x_{ij}\|$;
        \item[(3)] any sequence $x(1),x(2),\ldots$ in $M_n(x)$ converges to $x$ in $M_n(X)$ if and only if $x(1)_{ij},x(2)_{ij}$ converges to $x_{ij}$ in $X$ for each $i,j\in\{1,\ldots,n\}$. 
        \end{itemize}
    \end{lemma}

We can now prove a proposition that is analogous to Proposition
\ref{prop:separable Banach spaces are presentable} in $\Ban$.

\begin{proposition}
  \label{prop:separable-os}
  Let $X$ be a separable operator space, and let $D:\Lambda\to \OS$ be a countably-directed diagram with colimit $(C, \{ c_\lambda\}_{\lambda \in \Lambda})$ as given by
  Construction \ref{constr:aleph1-directed-colimit}. Let $f\colon X\to C$ be a complete contraction. Then
  \begin{itemize}
      \item[(1)] $f=c_\lambda\circ g$  for some $\lambda\in\Lambda$ and some complete contraction $g:X\to D_\lambda$;
      \item[(2)] for each $\lambda\in\Lambda$, if $g,g':X\to D_\lambda$ are bounded maps such that $c_\lambda\circ g=c_\lambda\circ g'$,
        then there exists some $\tau\geq\lambda$ such that $D_{\lambda,\tau}\circ g=D_{\lambda,\tau}\circ g'$.
  \end{itemize}
\end{proposition}
\begin{proof}
  For (1), let $A$ be a countable dense subset of $X.$
  For fixed $n\in\mathbb N$, Lemma \ref{lem:norm X vs norm MnX}.(3)
  implies that $\MM_n(A)$
  is a countable dense subset of $M_n(X)$. For each
  $n\in\mathbb N$, it follows from Lemma \ref{lem:amplification-colimits} that
  $M_n(C)$ is the colimit of the diagram
  $D_n:\Lambda\to\mathbf{Ban}$.
  Hence, if we apply Proposition \ref{prop:separable Banach spaces are presentable},
  we find some $\lambda_n\in\Lambda$ and some contraction
  $g^{(n)}:M_n(X)\to M_n(D_{\lambda_n})$ such that
  $c_{\lambda_n}^n\circ g^{(n)}=f_n$. Note that
  $c_{\lambda_n}^n$ and $f_n$ are the $n$-th amplifications of $c_{\lambda_n}$
  and $f$, respectively, but that $g^{(n)}$ is not necessarily such an
  amplification. By countable-directedness of $\Lambda$, there is some
  $\kappa\geq\lambda_n$ for each $n\in\mathbb N$.
  Let $h^{(n)}\eqdef D^n_{\lambda_n,\kappa}\circ g^{(n)} \colon M_n(X) \to M_n(D_\kappa) .$
  Then
  \[c_\kappa^n\circ h^{(n)}=c_\kappa^n\circ D_{\lambda_n,\kappa}^{n} \circ g^{(n)}=c_{\lambda_n}^n\circ g^{(n)}=f_n.\]
  Now, fix $n\in\mathbb N$. We can show that $(h^{(1)})_n\colon M_n(X)\to M_n(D_\kappa)$ is bounded.
  Indeed, for each $x=[x_{ij}] \in M_n(X)$, we have
  \begin{align*}
    \|(h^{(1)})_n(x)\| &= \|[h^{(1)}(x_{ij})]\| & \\
    &\leq\sum_{i,j=1}^n\|h^{(1)}(x_{ij})\| & (\text{Lemma \ref{lem:norm X vs norm MnX}.(2)}) \\
    &\leq\sum_{i,j=1}^n \|x_{ij}\| & (h^{(1)} \text{ is a contraction}) \\
    &\leq\sum_{i,j=1}^n\|x\| & (\text{Lemma \ref{lem:norm X vs norm MnX}.(1)}) \\ 
    &=n^2\|x\| &
  \end{align*}
  which shows that $(h^{(1)})_n$ is bounded.
  Furthermore,  because
  $c_\kappa\circ h^{(1)}=f$, we have $c_{\kappa}^n\circ (h^{(1)})_n=f_n$. Thus
  $h^{(n)}$ and $(h^{(1)})_n$ are bounded maps $M_n(X)\to M_n(D_\kappa)$ such
  that $c_{\kappa}^n\circ h^{(n)}=f_n=c_{\kappa}^n\circ (h^{(1)})_n$. It now
  follows from Proposition \ref{prop:separable Banach spaces are
  presentable}.(2) that there is some $\tau_n\geq\kappa$ such that
  $D_{\kappa,\tau_n}^n\circ h^{(n)}=D_{\kappa,\tau_n}^n\circ (h^{(1)})_n$. By
  countable-directedness of $\Lambda$, there is some $\lambda\geq\tau_n$ for
  each $n\in\mathbb N$. Define $g\colon X\to D_\lambda$ as $g\eqdef
  D_{\kappa,\lambda}\circ h^{(1)}$. Then for each $n\in\mathbb N$, we have
  \[
    g_n = D_{\kappa,\lambda}^n\circ (h^{(1)})_n
    = D_{\tau_n,\lambda}^n \circ D_{\kappa,\tau_n}^n \circ (h^{(1)})_n
    = D_{\tau_n,\lambda}^n \circ D_{\kappa,\tau_n}^n \circ h^{(n)}.
  \]
  Thus $g_n$ is a composition of contractions, therefore $g$ is a complete contraction.
  Moreover, we have
  \[ c_\lambda\circ g=c_\lambda\circ D_{\kappa,\lambda}\circ h^{(1)}=c_\kappa\circ h^{(1)}=f . \]

  For (2), if $g,g':X\to D_\lambda$ are bounded maps such that
  $c_\lambda\circ g=c_\lambda\circ g'$, we can apply Proposition \ref{prop:separable
  Banach spaces are presentable}.(2) to conclude the existence of some
  $\tau\geq\lambda$ such that
  $D_{\lambda,\tau}\circ g=D_{\lambda,\tau}\circ g'$.
\end{proof}

Next, we would like to prove that any countably-presentable operator space in
$\OS$ is separable. For this, our proof strategy is completely analogous to the
one we used in $\Ban.$ Note that, if $X$ is an operator space, then any closed
subspace of $X$ inherits an operator space structure from $X$ (see
\cite[Chapter 3.1]{effros-ruan}) and the inclusion of any such closed subspace
into $X$ is a complete isometry.

\begin{proposition}
  \label{prop:os-self-colimit}
  Every operator space is a countably-directed colimit (in $\OS$) of its closed separable subspaces.
\end{proposition}
\begin{proof}
  The proof is completely analogous to the proof of Proposition
  \ref{prop:ban-self-colimit}. Simply replace all instances of ``Banach space'',
  ``contraction'', and ``isometry' by ``operator space'', ``complete contraction'',
  and ``complete isometry'', respectively. 
\end{proof}

Finally, we can prove the main result of the subsection.

\begin{theorem}
  \label{thm:presentable-objects}
  An operator space $X$ is countably-presentable in $\OS$ iff $X$ is separable.
\end{theorem}
\begin{proof}
  $(\Leftarrow)$ Combine Proposition \ref{prop:separable-os} and Proposition
  \ref{prop:presentable-object}.

  $(\Rightarrow)$ Completely analogous to the proof of the same direction in
  Theorem \ref{thm:separable-banach}. 
\end{proof}

\subsection{Local Presentability of $\OS$}  
\label{sub:os-locally-presentable}

We may now combine our previous results in order to prove the main one.

\begin{theorem}
  \label{thm:os-locally-presentable}
  The category $\OS$ is locally countably presentable.
\end{theorem}
\begin{proof}
  The category $\OS$ is cocomplete (Proposition \ref{prop:cocomplete}) and it has a
  strong generator (Proposition \ref{prop:strong-generator}) consisting of
  countably-presentable objects (Theorem \ref{thm:presentable-objects}), so
  the result follows by Proposition \ref{prop:locally-presentable-by-generators}.
\end{proof}

Our development shows that the locally presentably structures of $\OS$ and $\Ban$ are closely related to each other.
Indeed, our proofs build upon results already established in $\Ban$ in order to derive the new results for $\OS$ and
many of the constructions in $\OS$ are completely analogous to those in $\Ban.$
In fact, we can say even more about this close relationship from a categorical
viewpoint. In the theory of locally $\alpha$-presentable categories, there is
an important role that is played by reflective subcategories that are
closed under $\alpha$-directed colimits. We invite the reader to consult
\cite{lp-categories-book} for more information. In fact, the relationship
between $\Ban$ and $\OS$ satisfies this.

To recognise this, recall that a \emph{full} subcategory $\CC$ of $\DD$ is reflective whenever its inclusion $\CC \hookrightarrow \DD$ has a left adjoint.
We have an adjunction (see \eqref{eq:triple-adjunctions})
\cstikz{min-adjunction.tikz}
and the functor $\Min$ is full and faithful. Since $\Min$ is a right
adjoint between locally countably presentable categories, it follows that it
preserves countably-directed colimits \cite[Proposition
2.23]{lp-categories-book}. Let $\mathbf{Min}$ be the full subcategory of $\OS$
consisting of minimal operator spaces, equivalently, operator spaces $X$ such
that $X = \Min(U(X)).$

\begin{corollary}
  We have an isomorphism of categories
  \[ \Ban \cong \mathbf{Min} , \]
  where $\mathbf{Min}$ is a reflective subcategory of $\OS$ closed under
  countably-directed colimits.
\end{corollary}
\begin{proof}
  The isomorphism is given by the (co)restriction of the functors $\Min$ and
  $U$ and the rest follows from the arguments presented above.
\end{proof}

\section{Cofree (Cocommutative) Coalgebras}
\label{sec:coalgebras}

In this section we discuss a corollary of the local presentability of $\OS$
relating to cofree (cocommutative) coalgebras. The local presentability of
$\OS$ behaves very well with respect to the \emph{operator space projective
tensor product} (see \cite[Chapter 7]{effros-ruan} and \cite[Section
1.5]{blecher-merdy}). The coalgebras under consideration here are taken with
respect to it. This tensor product is a completion of the algebraic tensor
product with respect to a suitable norm and such that it also enjoys an
important universal property.

Recall that, if $X$ and $Y$ are vector spaces over $\mathbb C$, then the
algebraic tensor product $X\otimes Y$ of $X$ and $Y$ enjoys the following
universal property: there exists a bilinear map $\pi\colon X\times Y\to
X\otimes Y$ such that for each vector space $Z$ and each bilinear map $u\colon
X\times Y\to Z$ there is a \emph{unique} linear map $\tilde u\colon X\otimes
Y\to Z$ such that $\tilde u\circ\pi= u$. We call $\tilde u$ the
\emph{linearization} of $u$. In particular, the linearization of $\pi$ itself
is the identity on $X\otimes Y$. We call elements in the image of $\pi$
\emph{elementary tensors}. In particular, for $x\in X$ and $y\in Y$, we write
$x\otimes y\eqdef \pi(x,y)$.

We can now recall the operator space projective tensor product $\ptimes$ and
the universal property that it enjoys.

\begin{definition}
  \label{def:projective-tensor-operator}
  Let $X$ and $Y$ be operator spaces. For an element $v \in \MM_n(X \otimes Y)$, consider the norm
  \begin{align*}
    \norm{v}_{\wedge} \eqdef \inf \{ \norm \alpha \norm x \norm y \norm \beta \  |\ & p \in \mathbb N, q \in \mathbb N, x \in M_p(X) , y \in M_q(Y), \alpha \in M_{n, pq}, \\
    & \beta \in M_{pq, n} , \text{ and } v = \alpha (x \otimes y) \beta\} .
  \end{align*}
  Here $x\otimes y\in M_{pq}(X\otimes Y)$ is the ``tensor product of matrices'' defined by
  \[x\otimes y\eqdef [x_{ij}\otimes y_{kl}]_{(i,k),(j,l)}.\]
  Note that the expression for the norm immediately yields $\|x\otimes y\|_{\wedge}\leq \|x\|\|y\|$ for such $x\otimes y\in \MM_{pq}(X\otimes Y)$. 
  Less obviously, it is also true that $\|x\otimes y\|_{\wedge} = \|x\|\|y\|.$
  We write $X \ptimes Y$ for the completion of $X \otimes Y$ with respect to the above norm (on $\MM_1(X \otimes Y)$) and say that $X \ptimes Y$ is the 
  operator space projective tensor product\footnote{This does not coincide with the Banach space projective tensor product of $X$ and $Y$. The two tensors are different in general.}
  of $X$ and $Y$. More specifically, $M_n(X \ptimes Y)$ is given by the completion of $\MM_n(X \otimes Y)$ with respect to the above norm and this determines the OSS of $X \ptimes Y.$
\end{definition}

A proof that $X\ptimes Y$ is an operator space can be found in
\cite[1.5.11]{blecher-merdy} or \cite[Section 7.1]{effros-ruan}. To formulate
the universal property of this tensor, we first recall the appropriate kinds of
bilinear maps.

\begin{definition}
  Let $X$, $Y$ and $Z$ be operator spaces. Then a bilinear map $u:X\times Y\to Z$ is called \emph{jointly completely bounded} if there exists a $K\geq 0$ such that for each $n,m\in\mathbb N$ and each $[x_{ij}]\in M_n(X)$ and each $[y_{kl}]\in M_m(Y)$, we have 
    \[ \|[u(x_{ij},y_{kl})]_{(i,k),(j,l)}\|\leq K\| [x_{ij}]\|\|[y_{kl}]\|.\]
  We define $\|u\|_{\mathrm{jcb}}$ to be the least $K$ satisfying the above condition.
  If $\|u\|_{\mathrm{jcb}}\leq 1$, we say that $u$ is \emph{jointly completely contractive}. 
  We write $\JCB(X\times Y, Z)$ for the space of all jointly completely bounded bilinear maps $X\times Y\to Z$. It follows that $\|\cdot\|_{\mathrm{jcb}}$ is a norm on $\JCB(X \times Y, Z)$ and it can also be seen as a matrix norm through the linear isomorphism
  \[ \mathbb M_n(\JCB(X \times Y, Z)) \cong \JCB(X \times Y, M_n(Z)) . \]
\end{definition}

Given two operator spaces $X$ and $Y$, the map
$\pi_{X,Y}\colon X\times Y\to X\otimes Y$ extends to a jointly completely
contractive map $X\times Y\to X\ptimes Y$, which we also denote by
$\pi_{X,Y}$, and which satisfies the following universal property.

\begin{proposition}\cite[1.5.11]{blecher-merdy}\cite[Proposition 7.1.2]{effros-ruan}
  \label{prop:universal-tensor}
  Given operator spaces $X$, $Y$, and $Z$, for any jointly completely
  contractive (bounded) map $u\colon X\times Y\to Z$ there is a unique
  completely contractive (bounded) map $\bar u:X\ptimes Y\to Z,$ such that
  \begin{equation}\label{eq:universal property}
    \bar u\circ\pi_{X,Y} =u,
  \end{equation}
  which is obtained as the unique continuous extension of
  $\tilde u:X\otimes Y\to Z$. Moreover, we have
  $\|u\|_{\mathrm{jcb}}=\|\bar u\|_{\mathrm{cb}}$. In particular, we have a completely
  isometric isomorphism
  \begin{align*}
    JCB(X\times Y,Z) &\cong \CB(X\ptimes Y,Z) \\ 
    u &\mapsto \bar u.
  \end{align*}
\end{proposition}

The following proposition captures all of the categorical structure and
properties of the operator space projective tensor product that we need. We
believe that it might be folklore knowledge and the monoidal closure of $\OS$
has been explicitly pointed out by Yemon Choi in an online discussion already
\cite{choi-smcc}. In presenting Proposition \ref{prop:os-smcc}, our intention
is not to claim originality of this result, but to sketch how this can be
proven.

\begin{proposition}
  \label{prop:os-smcc}
  The category $\OS$ has the structure of a symmetric monoidal closed category
  with respect to the operator space projective tensor product. More precisely:
  \begin{itemize}
    \item The monoidal product of operator spaces $X$ and $Y$ is given by $X \ptimes Y$;
    \item The monoidal unit is given by the complex numbers $\mathbb C$;
    \item The internal-hom between operator spaces $X$ and $Y$ is given by the operator space $\CB(X,Y)$;
    \item The left unitor $\lambda_X \colon \C \ptimes X \xrightarrow{\cong} X$ is the completely isometric
      isomorphism uniquely determined by the assignment $a \otimes x \mapsto ax$;
    \item The right unitor $\rho_X \colon X \ptimes \C \xrightarrow{\cong} X$ is the completely isometric
      isomorphism uniquely determined by the assignment $x \otimes a \mapsto ax$;
    \item The associator $\alpha_{X,Y,Z} \colon (X \ptimes Y) \ptimes Z \xrightarrow{\cong} X  \ptimes ( Y \ptimes Z)$ is the completely isometric
      isomorphism uniquely determined by the assignment $(x \otimes y) \otimes z \mapsto x \otimes (y \otimes z)$;
    \item The symmetry $\sigma_{X,Y} \colon X \ptimes Y \xrightarrow{\cong} Y \ptimes X$ is the completely isometric
      isomorphism uniquely determined by the assignment $x \otimes y  \mapsto y \otimes x $.
  \end{itemize}
\end{proposition}
\begin{proof} (Sketch).
  The construction of many of these isomorphisms is already outlined in
  \cite[1.5.11]{blecher-merdy} or \cite[Section 7.1]{effros-ruan}. The
  construction of the associator requires a bit more work in terms of universal
  properties, but we omit the details here. The tensor product $(\cdot \ptimes
  \cdot)$ can be easily extended to a bifunctor and it is easy to check the
  naturality of the above isomorphisms by reducing them to elementary tensors.
  The coherence equations for a symmetric monoidal category can also be easily
  checked by reducing them to elementary tensors, where they are obviously
  satisfied. For monoidal closure, we have completely isometric isomorphisms
  \[ \CB(X \ptimes Y, Z) \cong \JCB(X \times Y, Z) \cong \CB(X, \CB(Y,Z)) \qquad \cite[\text{Proposition 7.1.2}]{effros-ruan} \]
  and taking the restrictions to the unit balls gives
  \[ \OS(X \ptimes Y, Z) \cong \OS(X, \CB(Y,Z)) \]
  which is the required isomorphism on the external-homs. Verifying the
  naturality of these isomorphisms is straightforward.
\end{proof}

We can now describe the coalgebras that behave very well with respect to the
locally presentable structure of $\OS.$

\begin{definition}
  \label{def:coalgebra}
  A $\ptimes$-\emph{coalgebra} in $\OS$ is given by the following data:
  \begin{itemize}
    \item an operator space $X$,
    \item a complete contraction $c \colon X \to X \ptimes X$, called \emph{comultiplication},
    \item a complete contraction $d \colon X \to \C$, called \emph{counit},
  \end{itemize}
  such that the following equations hold (we suppress the left/right unitors and the associator for readability):
  \begin{align*}
    (d \ptimes \id_X) \circ c &= \id_X & \text{(left counitality)} \\
    (\id_X \ptimes d) \circ c &= \id_X & \text{(right counitality)} \\
    (c \ptimes \id_X) \circ c &= (\id_X \ptimes c) \circ c & \text{(coassociativity)}
  \end{align*}
  If, moreover, the following equation holds:
  \begin{align*}
    \sigma_{X,X} \circ c      &= c & \text{(cocommutativity)}
  \end{align*}
  then we say that the $\ptimes$-coalgebra $(X, c, d)$ is \emph{cocommutative}.
\end{definition}

Our results extend only to (cocommutative) coalgebras with respect to
$\ptimes$, so going forward we simply refer to them as (cocommutative)
coalgebras without explicitly mentioning the tensor product. These
kinds of coalgebras also form categories that we define next.

\begin{definition}
  Let $(X,c,d)$ and $(Y,c',d')$ be two coalgebras in $\OS$. A \emph{coalgebra morphism} is a complete contraction $f \colon X \to Y,$ such that:
  \begin{align*}
    (f \ptimes f)  \circ c& = c' \circ f \\
    d'  \circ f& = d
  \end{align*}
  We write $\Coalg$ for the category whose objects are the coalgebras in $\OS$ and whose morphisms are the coalgebra morphisms between them.
  We write $\CoCoalg$ for the full subcategory of $\Coalg$ consisting of cocommutative coalgebras.
\end{definition}

Using the already established results, we can now conclude that these two categories are quite nice.

\begin{theorem}
  The category $\Coalg$ is locally presentable and symmetric monoidal closed. The category $\CoCoalg$ is locally presentable and cartesian closed.
  Furthermore, we have adjunctions
  \begin{equation}
    \label{eq:cofree}
    \stikz{free-coalgebra-noncommutative-adjunction.tikz}
  \end{equation}
  and
  \begin{equation}
    \label{eq:cocommutative-cofree}
    \stikz{free-coalgebra-adjunction-os.tikz}
  \end{equation}
  where the \emph{left} adjoints $V$ and $U$ are the forgetful functors.
\end{theorem}
\begin{proof}
  The category $\OS$ is locally presentable (Theorem
  \ref{thm:os-locally-presentable}) and symmetric monoidal closed (Proposition
  \ref{prop:os-smcc}), so the proof follows using purely categorical arguments,
  see \cite[pp. 10 and pp. 13]{porst-coalgebras} and note that the notion of
  ``comonoid'' in that paper coincides with our use of ``coalgebra''.
\end{proof}

Adjunction \eqref{eq:cofree} is perhaps of greater interest to operator
algebraists, whereas adjunction \eqref{eq:cocommutative-cofree} could be of
greater interest to some logicians, because it is a model of Intuitionistic
Linear Logic \cite{linear-logic} in the sense of Lafont \cite{lafont-thesis}.
If we unpack the structure of these adjunctions, they show us
that the \emph{cofree (cocommutative) coalgebra} over an operator space $X$
always exists. Let us explain what this means in more detail
(see \cite[\S 7.2]{mellies-linear-logic} for more information).

If $X$ is an operator space, then $GX$ is a coalgebra for which there exists a canonical complete contraction
$\epsilon_X \colon GX \to X$, such that for any coalgebra $C$ and any complete contraction $f \colon C \to X$,
there exists a \emph{unique} coalgebra morphism $\hat f \colon C \to GX$, such that the following diagram
\cstikz{cofree-property.tikz}
commutes. The cofree cocommutative coalgebra may be described in a similar way.

\section{Conclusion}
\label{sec:conclusion}

We showed that the category $\OS$ is locally countably presentable and we
characterised the countably-presentable objects as exactly the separable
operator spaces. In the process, we also showed that the separable Banach
spaces coincide with the countably-presentable objects in $\Ban$. Our main
results are not surprising, but they do require some lengthy technical effort
related to the construction of countably-directed colimits. We hope that the
result related to cofree (cocommutative) coalgebras can illustrate the utility
of categorical methods when studying operator spaces.  The right adjoints, $G$
and $R$, are shown to exist via the adjoint functor theorem, which relies on
the axiom of choice, and we think that proving the existence of these cofree
(cocommutative) coalgebras can be difficult to achieve with more traditional
methods from functional analysis and operator space theory. By using methods
from category theory, we were able to prove their existence by relying on
already established categorical results.

Thea Li completed an internship project \cite{thea-internship} under the
supervision of the second author where she proved that the category of
\emph{finite-dimensional} operator spaces and complete contractions is
$*$-autonomous, it is a model of multiplicative additive linear logic, and it
also has the structure of a BV-category \cite{bv-category} with respect to the
Haagerup tensor product.  As part of future work, we would like to identify a
\emph{subcategory} of $\OS$ which is also $*$-autonomous but with the
additional property of being a model of (full) classical linear logic in the
sense of Lafont, i.e. where we again have an adjoint situation with the
category of cocommutative coalgebras like in \eqref{eq:cocommutative-cofree}.

\mbox{}

\noindent\textbf{Acknowledgements.} We thank Thea Li, James Hefford, and
Jean-Simon Pacaud Lemay for discussions. 

\newpage

\bibliographystyle{plain}
\bibliography{refs}

\end{document}